\documentclass[a4paper,11pt]{amsart}
\pdfoutput=1
\usepackage{amssymb,amsmath,amsthm,amsfonts,centernot}
\usepackage[colorlinks=true,linkcolor=blue,citecolor=blue,urlcolor=cyan, pdfpagelabels=false]{hyperref}
\usepackage{amsmath,tikz}
\usepackage{float}
\usepackage{dsfont}
\usepackage{color}
\usetikzlibrary{matrix}
\usepackage{perpage} 
\usepackage{makecell}
\usepackage{threeparttable}
\MakePerPage{footnote}
\begin{document}
\numberwithin{equation}{section}
\newtheorem{theorem}{Theorem}
\newtheorem{algo}{Algorithm}
\newtheorem{lem}{Lemma} 
\newtheorem{de}{Definition} 
\newtheorem{ex}{Example}
\newtheorem{pr}{Proposition} 
\newtheorem{claim}{Claim} 
\newtheorem{re}{Remark}
\newtheorem{co}{Corollary}
\newtheorem{conv}{Convention}
\newcommand{\di}{\hspace{1.5pt} \big|\hspace{1.5pt}}
\newcommand{\idi}{\hspace{.5pt}|\hspace{.5pt}}
\newcommand{\hs}{\hspace{1.3pt}}
\newcommand{\thmf}{Theorem~1.15$'$}
\newcommand{\ndi}{\centernot{\big|}}
\newcommand{\nidi}{\hspace{.5pt}\centernot{|}\hspace{.5pt}}
\newcommand{\lp}{\mbox{$\hspace{0.12em}\shortmid\hspace{-0.62em}\alpha$}}
\newcommand{\PQ}{\bb{P}^1(\bb{Q})}
\newcommand{\pmn}{\cl{P}_{M,N}}
\newcommand{\lcm}{\operatorname{lcm}}
\newcommand{\he}{holomorphic eta quotient\hspace*{2.5pt}}
\newcommand{\hes}{holomorphic eta quotients\hspace*{2.5pt}}
\renewcommand{\phi}{\varphi}
\newcommand{\Z}{\bb{Z}}
\newcommand{\ZD}{\Z^{\bb{D}_N}}
\newcommand{\N}{\bb{N}}
\newcommand{\Q}{\bb{Q}}
\newcommand{\pii}{{{\pi}}}
\newcommand{\R}{\bb{R}}
\newcommand{\C}{\bb{C}}
\newcommand{\I}{\hs\cl{I}_{n,N}}
\newcommand{\SL}{\operatorname{SL}_2(\Z)}
\newcommand{\St}{\operatorname{Stab}}
\newcommand{\D}{\bb{D}_N}
\newcommand{\rh}{{{\boldsymbol\rho}}}
\newcommand{\bh}{{\cl{M}}} 
\newcommand{\lv}{\hyperlink{level}{{\text{level}}}\hspace*{2.5pt}}
\newcommand{\fct}{\hyperlink{factor}{{\text{factor}}}\hspace*{2.5pt}}
\newcommand{\q}{\hyperlink{q}{{\mathbin{q}}}}
\newcommand{\rd}{\hyperlink{redu}{{{\text{reducible}}}}\hspace*{2.5pt}}
\newcommand{\ird}{\hyperlink{irredu}{{{\text{irreducible}}}}\hspace*{2.5pt}}
\newcommand{\str}{\hyperlink{strong}{{{\text{strongly reducible}}}}\hspace*{2.5pt}}
\newcommand{\rdn}{\hyperlink{redon}{{{\text{reducible on}}}}\hspace*{2.5pt}}
\newcommand{\atl}{\hyperlink{atinv}{{\text{Atkin-Lehner involution}}}\hspace*{3.5pt}}
\newcommand{\atls}{\hyperlink{atinv}{{\text{Atkin-Lehner involutions}}}\hspace*{3.5pt}}
\newcommand{\T}{\mathrm{T}}
\renewcommand{\H}{\fr{H}}
\newcommand{\W}{\text{\calligra W}_n}
\newcommand{\GG}{\cl{G}}
\newcommand{\g}{\fr{g}}
\newcommand{\Gm}{\Gamma}
\newcommand{\Gmtl}{\widetilde{\Gamma}_\ell}
\newcommand{\gm}{\gamma}
\newcommand{\go}{\gamma_1}
\newcommand{\gmt}{\widetilde{\gamma}}
\newcommand{\gmdt}{\widetilde{\gamma}'}
\newcommand{\gmot}{\widetilde{\gamma}_1}
\newcommand{\gmdot}{{\widetilde{\gamma}}'_1}
\newcommand{\s}{\Large\text{{\calligra r}}\hspace{1.5pt}}
\newcommand{\ms}{m_{{{S}}}}
\newcommand{\nisim}{\centernot{\sim}}
\newcommand{\level}{\hyperlink{level}{{\text{level}}}}
\newcommand{\Redcon}{the \hyperlink{red}{\text{Reducibility~Conjecture}}}
\newcommand{\ReDcon}{The \hyperlink{red}{\text{Reducibility~Conjecture}}}
\newtheorem*{ThmA}{Theorem A}
\newtheorem*{ThmB}{Theorem B}
\newtheorem*{ThmC}{Theorem C}
\newtheorem*{ThmD}{Theorem D}
\newtheorem*{ThmE}{Theorem E}
\newtheorem*{ThmF}{Theorem F}
\newtheorem*{LmA}{Lemma A}
\newtheorem*{CorA}{Corollary A}
\newtheorem*{CorB}{Corollary B}
\newtheorem*{CorC}{Corollary C}
\newcommand{\Conred}{Conjecture~$1'$}
\newcommand{\effth}{Theorem~$\ref{27.11.2015}'$}
\newcommand{\Conredd}{Conjecture~$1''$}
\newcommand{\Conreddd}{Conjecture~$1'''$}
\newcommand{\Conired}{Conjecture~$2'$}
\newtheorem*{cau}{Caution}
\newtheorem{thrmm}{Theorem}[section]
\newtheorem*{thmA}{Theorem~A}
\newtheorem*{corA}{Corollary~A}
\newtheorem*{corB}{Corollary~B}
\newtheorem*{corC}{Corollary~C}
\newtheorem{no}{Notation}
\renewcommand{\thefootnote}{\fnsymbol{footnote}}
\newtheorem{oq}{Open question}
\newtheorem{conj}{Conjecture}
\newtheorem{hy}{Hypothesis} 
\newtheorem{expl}{Example}
\newcommand\ileg[2]{\bigl(\frac{#1}{#2}\bigr)}
\newcommand\leg[2]{\Bigl(\frac{#1}{#2}\Bigr)}
\newcommand{\e}{\eta}
\newcommand{\sgn}{\operatorname{sgn}}
\newcommand{\bb}{\mathbb}
\newtheorem*{conred}{Conjecture~\ref{con1}$\mathbf{'}$}
\newtheorem*{conredd}{Conjecture~\ref{con1}$\mathbf{''}$}
\newtheorem*{conreddd}{Conjecture~\ref{con1}$\mathbf{'''}$}
\newtheorem*{conired}{Conjecture~\ref{19.1Aug}$\mathbf{'}$}
\newtheorem*{efth}{Theorem~\ref{27.11.2015}$\mathbf{'}$}
\newtheorem*{eflem}{Theorem~\ref{15May15}.$(b)\mathbf{'}$}
\newtheorem*{procl}{\textnormal{\textit{Proof}}}
\newtheorem*{thmff}{Theorem~\ref{7.4Jul}$\mathbf{'}$}
\newtheorem*{coo}{Corollary~\ref{17Aug}$\mathbf{'}$}
\newcommand{\cooo}{Corollary~\ref{17Aug}$'$}
\newtheorem*{cotw}{Corollary~\ref{17.1Aug}$\mathbf{'}$}
\newcommand{\cotww}{Corollary~\ref{17.1Aug}$'$}
\newtheorem*{cothr}{Corollary~\ref{17.2Aug}$\mathbf{'}$}
\newcommand{\cothre}{Corollary~\ref{17.2Aug}$'$}
\newtheorem*{cne}{Corollary~\ref{15.5Aug}$\mathbf{'}$}
\newcommand{\cnew}{Corollary~\ref{15.5Aug}$'$\hspace{3.5pt}}
\newcommand{\fr}{\mathfrak}
\newcommand{\cl}{\mathcal}
\newcommand{\rad}{\mathrm{rad}}
\newcommand{\ord}{\operatorname{ord}}
\newcommand{\m}{\setminus}
\newcommand{\G}{\Gamma_1}
\newcommand{\GN}{\Gamma_0(N)} 
  \newcommand{\un}{\hs\underline{\hspace{5pt}}\hs}

\title[The structure of irreducible cuspidal eta quotients]{The structure of
irreducible \\
cuspidal eta quotients}
\author{Soumya Bhattacharya}
\address
{Indian Institute of Science Education and Research Kolkata} 
\email{soumya.bhattacharya@iiserkol.ac.in}

 \begin{abstract}
We show that an ir{\-}reducible cuspidal eta quotient 
of level~$N$ is uniquely identified by its orders of vanishing at 
$O((\log_2 N)^{r-1})$ cusps, where 
$r$ denotes the number of distinct prime divisors of $N$.
In particular, 
all such eta quotients 
which are Fricke eigenforms of a prime power level,
are uniquely identified only
by their orders of vanishing at~$\infty$. 
We also prove that for all $k\in\N$, there exists a primitive and irreducible 
cuspidal eta quotient of weight $k/2$. It follows from the exhaustiveness of Zagier's list of the primitive holomorphic eta quotients of weight~$1/2$ that the 1st Mersmann level, $\cl{M}_1=12$. Here we show that for all $k\in\N$, the $k$-th Mersmann level is divisible by $2^{2k}$.
 \end{abstract}
 
 \maketitle
 
 \section{Introduction} 
 \label{intro}
 An eta quotient $f$ is a finite product of the form 
\begin{equation}
\prod\e(dz)^{X_d},\label{13.04.2023}
\end{equation}
where $d\in\N$, $X_d\in\Z$ and $\e$ is
the Dedekind eta function defined by 
 \begin{equation}
 \eta(z)=e^{\frac{\pi iz}{12}}\prod_{n=1}^\infty(1-e^{2\pi inz})
\nonumber\end{equation} 
for all $z$ in the complex upper half-plane $\H$.
Eta quotients occur in several areas of Mathematics (see \cite{ak, coh, z} 
for an overview). Since eta is a modular form of weight~$1/2$ with a multiplier system on $\operatorname{SL}_2(\Z)$ (see 
\cite{b}), the eta quotient $f$ in
(\ref{13.04.2023}) transforms like a modular form of weight $\frac12\sum_dX_d$ with a multiplier system on suitable
congruence subgroups of $\operatorname{SL}_2(\Z)$. The largest among these subgroups is 
\begin{equation}
\Gm_0(N):=\Big{\{}\begin{pmatrix}a&b\\ c&d\end{pmatrix}\in\operatorname{SL}_2(\Z)\hspace{3pt}\Big{|}\hspace{3pt}
c\equiv0\hspace*{-.3cm}\pmod N\Big{\}},
\nonumber\end{equation}
where $N:=\operatorname{lcm}\{d\in\N\hs\idi\hs X_d\neq0\}.$ 
We call $N$ the level of $f$. \emph{An eta quotient on} $\GN$ is an eta quotient whose level divides $N$.
Since $\eta$ is nonzero on $\H$, the eta quotient on $\GN$ is holomorphic if and only if
it does not have any pole at the cusps of $\GN$.  
A minimal set of representatives of the cusps of $\GN$ is given by
\begin{equation}
\Big{\{}\frac{a}{t}\text{ \ such that \ } t|N \text{ \  \& \ } \gcd(a,t)=1\Big{\}}/\sim\hspace{1.5pt} ,
\label{minset}
\end{equation}
where $\frac{a}{t}\sim\frac{b}{t}$ \ if and only if \ $a\equiv b\pmod{\gcd(t,N/t)}$ (see \cite{ds}). For ${a}/t$ as in \eqref{minset} and for an eta quotient ${f=\underset{d|N}{\prod}\eta(dz)^{X_d}}$ on $\GN$, we have
\begin{equation}
  \mathrm{ord}_{{a}/{t}}(f\hspace{1pt};\Gamma_0(N)) \   = \  \frac{N}{24\gcd(t,N/t)}\hspace{1pt}\sum_{d|N}\frac{\gcd(t,d)}{\lcm(t,d)}X_d
\label{27.04.2015.1}\end{equation}
(see \cite{ymart}). We define the \emph{cusp classes} of $\GN$ by
\begin{equation}\bigg\{\text{The cusps represented by }\frac{a}t\text{ for all $a\in\Z$ with}\gcd(a,t)=1\bigg\}_{t|N}.
\label{14.12.2025}\end{equation}
By the \emph{order of vanishing} of the eta quotient $f$ at a cusp class
of $\GN$, we mean order of vanishing of $f$ at any cusp in that cusp class (see \eqref{27.04.2015.1}). 
The reciprocals of the divisors of $N$ serve as the canonical representatives of the cusp classes
of $\GN$.
An eta quotient $f$ on $\GN$ is uniquely determined by its orders of vanishing 
at the cusp classes of $\GN$ (see \cite{B1, km}) and
the \emph{holomorphy} of $f$ is equivalent to the non{\-}negativity of its orders 
at all the cusp classes of $\GN$.

There are only finitely many holomorphic eta quotients of a given weight and a given level (see Theorem~3.9 in \cite{b}).
It is also well-known that every modular function on $\GN$ is a rational function of eta quotients on $\Gamma_0(4N)$ (see \cite{KF}).
Ono in \cite{on},
asks for a classification of 
the spaces of modular forms
which are generated by eta quotients. 
We note that if $N$ is sufficiently larger than
the number of its divisors, then the space of modular forms of a given weight on $\GN$
does not contain enough holomorphic eta quotients to constitute a basis
(see \S3 of \cite{B1} or \S2 of \cite{rw}). 

\begin{de}[Irreducible holomorphic eta quotients]
\textnormal{Let $f$ and $g\notin\{1,f\}$ be two holomorphic eta quotients.
If the eta quotient $f/g$ is holomorphic,
we call $g$ a \emph{proper factor} of $f$.
A holomorphic eta quotient $f$ is \emph{reducible} 
if it has a proper factor. Otherwise, we call $f$ \emph{ir{\-}reducible}}.
\label{10:48AM}\end{de}
An eta quotient $f$ is \emph{primitive} if 
it is not a rescaling of another eta quotient 
by a positive integer, i.e. 
no eta quotient $h$ and no integer $\nu>1$ satisfy the equation $f(z)=h(\nu z)$, 
for all $z\in\H$. 
Primitive and irreducible holomorphic eta quotients were first considered by Zagier, who 
conjectured \cite{z}:
\hypertarget{MF}{}
\begin{align}
&\text{\emph{There are only finitely many primitive and irreducible holomorphic}}\label{26.11.2015.A}\\[-.29em]
&\text{\emph{eta quotients of a given weight.}}\nonumber 
\end{align} 
Zagier's conjecture was proved by Mersmann in \cite{c}. 
For 
$k\in\N$, by ${\cl{M}_k}$ we denote the least common multiple of the levels of all the primitive and irreducible holomorphic eta quotients of weights less than or equal to~$k/2$ and we call it the \emph{$k$-th Mersmann level}.
Mersmann also proved\footnote{Mersmann proved the exhaustiveness of a list of fourteen primitive holomorphic eta quotients of weight 1/2 and levels 1,2,4,6 and 12, as conjectured by Zagier (see~\cite{z}).}
that $\cl{M}_1~=~12$ (see \cite{B2, c, z}). 
However, nothing except for a couple of divisibility results\footnote{obtained by extensive brute-force computations} for $\cl{M}_2$ and $\cl{M}_3$ (see Corollary~A.1 and Corollary~B.1 in \cite{B0})
are known for $\cl{M}_k$ with $k>1$. 
In the following, we conclude from Corollary~\ref{snld} that $2^{2k}|\cl{M}_k$ for all $k\in\N$. 
Corollary~\ref{12:42PM} provides a similar divisibility result in general, for all suitable primes. 

In analogy with \eqref{26.11.2015.A}, we obtained
the following result in~\cite{B3}:
\begin{ThmA}
There are only finitely many irreducible holomorphic eta quotients of a given level.
\end{ThmA}
Corollary~1 in \cite{B3} provides an upper bound on the weights of irreducible holomorphic eta quotients of level $N$
and also implies that 
the number of such eta quotients
is bounded above by $\rad(N)^{2\mathrm{d}(N)}$, where $\rad(N)$ and $\mathrm{d}(N)$ denote the 
product of distinct prime divisors and the number of divisors of~$N$, respectively. 

If a holomorphic
eta quotient $f$ has a proper factor $g$ such that both $g$ and $f/g$ are eta quotients on $\GN$, then 
$f$ is \emph{factorizable on} $\GN$.
The notion of reducibility of a holomorphic eta quotient has no apparent restriction
on the size of the levels of the factors. For example, 
\begin{equation}
\frac{\eta(z)\eta(2z)\eta(6z)}{\eta(3z)} \  = \ \frac{\eta(z)\eta(4z)\eta(6z)^2}
{\eta(2z)\eta(3z)\eta(12z)}\times\frac{\eta(2z)^2\eta(12z)}{\eta(4z)\eta(6z)},\label{eq1}\end{equation}
where a holomorphic eta quotient of weight 1 and level 6 is factored into two holomorphic eta quotients of weight 1/2 and level 12. 
It turns out that Mersmann's finiteness theorem \eqref{26.11.2015.A} implies,
the levels of all the factors of a holomorphic eta quotient $f$ are bounded above in terms of the weight and the level of $f$ (see Theorem~A in \cite{B4}):
\begin{ThmB}
For an integer $k\geq 2$, the levels of the factors of a holomorphic eta quotient $f$ of level $N$ and weight $k/2$ divide $\lcm(N, \cl{M}_{k-1}).$
\end{ThmB}

Since $\cl{M}_1=12$, we obtain the following:
\begin{corC}
The levels of the factors of a holomorphic eta quotient of weight $1$ and level $N$ divide $\lcm(N, 12)$.
\end{corC}
In particular, no factor of an eta quotient of weight 1 and level 6 can have a 
level that is 
higher than $12$ (see~\eqref{eq1}).
However, as mentioned above,
except for its existence and some divisibility results,
nothing is known about
$\cl{M}_k$ for $k>1$. Nevertheless, we do have an explicit upper bound on 
the minimum level of the proper factors of holomorphic eta quotients of a given level
(see Theorem~1 in \cite{B4}):
\begin{ThmD}
Let $f$ be a reducible holomorphic eta quotient  of level $N>1$  and
let $M$ be the minimum of the 
levels  of the proper factors of $f$. Then each prime
divisor of $M$  also divides $N$ and 
$$M<N^{2^{O(N^{2+\varepsilon})}},$$
where the implied 
constant is effectively computable. 
\end{ThmD}
Though the above theorem provides an 
upper bound on the minimum level of the proper factors of holomorphic eta quotients of a given level, this is
however far larger than the size speculated by Conjecture~1 in \cite{B4}, which has yet only been established for prime power levels (see Corollary~4 in \cite{B5}):
\begin{ThmE}
Each reducible holomorphic eta quotient of a prime power level $N$ has a proper factor whose level divides $N$. 
\hypertarget{E}{\color{white}{E}}\end{ThmE}

Lemma~7 in \cite{B3} also provides us with a nontrivial example of an irreducible holomorphic eta quotient of a given level, e.g. 
\begin{equation}
    \frac{\left(\eta(p^{m-2}q^{n-2}z)\hspace*{1pt}\eta(p^{m-2}q^nz)\hspace*{1pt}\eta(p^mq^{n-2}z)\hspace*{1pt}\eta(p^mq^nZ)\right)^{pq}\eta(p^{m-1}q^{n-1}z)^{(p^2+1)(q^2+1)}}
{\left(\eta(p^{m-2}q^{n-1}z)\hspace*{1pt}\eta(p^mq^{n-1}z)\right)^{p(q^2+1)}\left(\eta(p^{m-1}q^{n-2}z)\hspace*{1pt}\eta(p^{m-1}q^nz)\right)^{q(p^2+1)}}
\nonumber\end{equation}
is an irreducible holomorphic eta quotient of level $p^mq^n$, for all primes $p,q$
and for all integers $m,n\geq2$. 
The above eta quotient 
vanishes only at the cusp class of $\Gamma_0(p^mq^n)$ which is represented by $1/(p^{m-1}q^{n-1})$. 
This is an example of an \emph{axial} eta quotient, which we define in the following.

\begin{de}[\emph{Axial, generic and cuspidal eta quotients}]
\textnormal{A holomorphic eta quotient on $\Gamma_0(N)$ is called axial
if it vanishes only at a single cusp class 
of $\Gamma_0(N)$.
Else, it is called a generic holomorphic eta quotient on $\Gamma_0(N)$. If an eta quotient $f$ on $\GN$ vanishes at all the cusps of $\GN$, then $f$ is called a cuspidal eta quotient.}
\end{de}
For example, the axial irreducible holomorphic eta quotients on $\Gm_0(4)$ are
$$\frac{\eta(z)^2}{\eta(2z)}\hs,\hs\frac{\eta(2z)^5}{\eta(z)^2\eta(4z)^2}\hs\ \text{and} \ \hs\frac{\e(4z)^2}{\e(2z)}$$
which vanish only at the cusp classes of 1, 1/2 and 1/4 of $\Gm_0(4)$, respectively, whereas
the generic irreducible holomorphic eta quotients on $\Gm_0(4)$ are 
$$\eta(z)\hs,\e(2z)\hs,\e(4z)\hs,\hs\frac{\eta(2z)^3}{\eta(z)\eta(4z)}\hs,\hs\frac{\eta(2z)^2}{\eta(z)}\hs,\frac{\eta(2z)^2}{\eta(4z)}\hs \text{and} \ \hs\frac{\eta(z)\eta(4z)}{\eta(2z)},$$
among which only the first four are cuspidal.
For $N\in\N$, all possible axial
irreducible holomorphic eta quotients on $\GN$
are known (see Lemma~6 and Lemma~7 in \cite{B3}).
In Corollary~\ref{17:02PM}, we obtain that
a~generic irreducible holomorphic eta quotient on $\GN$ is uniquely identified by
its orders of vanishing only at a
subset of the  cusp classes~of~$\GN$.

\begin{de}[Special set of cusp classes] 
\textnormal{
Given an integer $N>1$, let 
$N=\prod_{p|N}p^{n_p}$ be the prime factorization of $N$. Let $\D$ denote the set of all divisors of $N$.
Let $\{\alpha_p\}_{p|N}$
be a sequence with $\alpha_p\in\{0,1\}$ for all $p$
such that $\alpha_p=1$ if $n_p=1$. Let
\begin{equation}
 \Lambda_N:=\Bigg\{\prod_{p|N}p^{\beta_p}\in\D\ : \text{ 
 $\exists 
 \ p_0|N$ with} \ \beta_{p_0}\in\{\alpha_{p_0}, n_{p_0}\}\Bigg\}.
\label{12:00PM}\end{equation}
A special set of cusp classes of $\GN$ contains only the cusp classes represented by the reciprocals of the elements of $\Lambda_N$.
}
\label{9:22AM}\end{de}

\begin{table}[H]
\centering
\caption{\mbox{Cardinalities of the sets of unique identifiers of modular objects}\\}
\begin{threeparttable}
  \begin{tabular}{| c | c | c |} 
\hline&&\\
\textbf{Modular objects}
&\makecell{\textbf{Unique} \\ \textbf{identifiers}}&\makecell{\textbf{Cardinality}\\\textbf{of the set}\\\textbf{of identifiers}}
\\\hline
\makecell{A modular form $f$ of\\ weight $k$ and level $N$} &\makecell{The $q$-series coefficients of $f$ \\ up to the Strum bound}
&\makecell{\\ $\displaystyle{\bigg{\lfloor}\frac{k[\operatorname{SL}_2(\Z):\GN]}{12}\bigg{\rfloor}}$\\ \\$\displaystyle{=\bigg{\lfloor}\frac{k}{12}\prod_{p\idi N}(p^{n_p}+p^{n_p-1})\bigg{\rfloor}}$}
\\\hline
\makecell{A normalized \\ modular form $f$\\ of level $N$}  &\makecell{The orders of vanishing of $f$ \\ at the cusps of $\GN$} &  \makecell{\\ $\displaystyle{\sum_{d\idi N}\varphi((d,N/d))}$\\ \\
$\displaystyle{=\prod_{p\idi N}(p^{\lfloor\frac{n_p}{2}\rfloor}+p^{\lceil\frac{n_p}{2}\rceil-1})}$}
\\\hline
\makecell{An eta quotient $f$\\ of level $N$}  &\makecell{\\ The orders of vanishing of $f$ \\ 
at the cusp classes of $\GN$ \\ 
\\} &$\displaystyle{\prod_{p\idi N}(n_p+1)}$
\\\hline
\makecell{A generic irreducible \\holomorphic eta quotient $f$ \\ of level $N$}&\makecell{ \\ The orders of vanishing of $f$\\at a special set of cusp\tnote{*} \\ classes of $\GN$ \\ \\ }&\makecell{$\displaystyle{\prod_{p\idi N}(n_p+1)-\prod_{\substack{p\idi N\\ n_p>1}}(n_p-1)}$}
\\
\hline
\makecell{\\ A generic irreducible \\holomorphic eta quotient $f$ \\ which is an Atkin-Lehner \\ eigenform of level $N$\tnote{$\dagger$}\\ \ }&\makecell{ \\ The orders of vanishing of $f$\\at a special set of cusp \\ classes of $\GN$\\modulo the Atkin-Lehner\\involutions of level $N$\\ \ }&\makecell{$\displaystyle{\underset{p|N}{\prod}\left\lceil\frac{(n_p+1)}2\right\rceil-\underset{p|N}{\prod}\left\lceil\frac{(n_p-1)}2\right\rceil}$}
\\
\hline
\makecell{ \\ A generic irreducible \\holomorphic eta quotient $f$\\ of a prime power level $N$ \\ \ }&\makecell{The orders of vanishing of $f$\\
at the cusps 0 and $\infty$\tnote{\#}}&$2$\\
\hline
\makecell{ \\ A generic irreducible \\holomorphic eta quotient $f$\\ which is 
a Fricke eigenform\\ of a prime power level $N$
\\ \ }&\makecell{The order of vanishing of $f$\\
at
$\infty$\tnote{\S} }&1\\
\hline
\end{tabular}
\begin{tablenotes}
\footnotesize
\item[*] See Theorem~\ref{sth}.
\item[$\dagger$] Such irreducible holomorphic eta quotients do not exist unless $n_p>1$
for each prime $p|N$.
\item[\#]Instead of 0, we may also take $1/p$.
\item[\S]Instead of $\infty$, here we may take $0$ or $1/p$.
\end{tablenotes}
\end{threeparttable}
\end{table}

\section{The main results}
\label{mr}
\begin{theorem}Let $S_N$ be a special set of cusp classes of $\GN$. 
Each generic holomorphic eta quotient on $\GN$ which is not factorizable on $\GN$
is uniquely identified by its orders of vanishing at the cusp classes in $S_N$.
\label{sth}\end{theorem}
\begin{co}
A generic irreducible holomorphic eta quotient\footnote{In particular, each irreducible cuspidal eta quotient on $\GN$ is uniquely identified by its orders of vanishing at only
$\lfloor2(\log_2 N)^{r-1}\rfloor$ cusps of $\GN$.} on $\GN$ 
is uniquely identified by its orders of vanishing at only
$\lfloor2(\log_2 N)^{r-1}\rfloor$ cusps
of $\GN$, where 
$r$ denotes the number of distinct prime divisors of $N$.
\label{17:02PM}\end{co}
\begin{co}\label{ppi}
A generic irreducible holomorphic eta quotient of a prime power level is uniquely identified by its orders of vanishing
at the cusps $0$ and $\infty$.
\end{co}

A divisor $\delta$ 
\emph{unitarily divides} $N$ (denoted by $\delta\|N$) if $\gcd(\delta,N/\delta)=1.$ 
For an unitary divisor $\delta$ of $N$, 
the \emph{Atkin-Lehner involution}\footnote{see \cite{al} and \S1.6 in \cite{B0}.} $\mathcal{AL}_{\delta,N}$ of the eta quotients on $\GN$
is defined by
\begin{equation}
\mathcal{AL}_{\delta,N}\Big{(}\prod_{d|N}\eta(dz)^{X_d}\Big{)}  \  := \  \prod_{d|N}\eta\Big(\frac{\lcm(d,\delta)}{\gcd(d,\delta)}z\Big)^{X_d}.
\label{11:59AM}\end{equation} 
Atkin-Lehner involutions preserve 
 holomorphy and cuspidality. 
The Atkin-Lehner involution $\mathcal{AL}_{N,N}$
is called the \emph{Fricke involution} of level $N$.
An eta quotient of level~$N$ which is invariant 
under the Fricke involution of level $N$ (resp. all the Atkin-Lehner involutions of level $N$) is called a \emph{Fricke eigenform} (resp. an \emph{Atkin-Lehner eigenform}) of level~$N$.
\begin{co}
A generic irreducible holomorphic eta quotient 
which is also an Atkin-Lehner eigenform of level $N={\prod_{p|N}}p^{n_p}$,
is uniquely determined by its orders of vanishing at only
$$\underset{p|N}{\prod}\left\lceil\frac{(n_p+1)}2\right\rceil-\underset{p|N}{\prod}\left\lceil\frac{(n_p-1)}2\right\rceil$$
cusps of $\GN$. 
\label{a-l}\end{co}

\begin{co}Let $f$ be a generic irreducible holomorphic eta quotient and a Fricke eigenform
of a prime power level. Then $f$ is uniquely determined by its order of vanishing
at $\infty$.
\label{fi}\end{co}

All the above corollaries in particular, hold for irreducible cuspidal eta
quotients.
By $\eta_d$ we denote the \emph{rescaling} of $\eta$ by $d$,
defined as 
$$\eta_d(z):=\eta(dz)\text{ \ for all $z\in\mathfrak{H}$.}$$

For a prime $p$ and an $n\in\N$,
define the  eta quotient $\mathcal{H}_{p,n}$ by
    \begin{equation}
\mathcal{H}_{p,n}:=\begin{cases}
{\frac{\displaystyle{(\eta_p\hspace{1pt}\eta_{p^{n-1}})^p\hspace{1pt}\prod_{s=1}^{n-1}\eta_{p^s}^{(p-1)(p-2)}}}{\ \quad\displaystyle{(\eta\hspace{1pt}\eta_{p^n})^{p-1}}}}&\text{ if $n$ is odd}\\
\\
{\frac{\displaystyle{\eta_p^p\hspace{1pt}\eta_{p^{n-1}}^{(p-1)^2}\hspace{1pt}\prod_{s=1}^{n/2-1}\eta_{p^{2s-1}}^{(p-1)^2-p}\hspace{1pt}\eta_{p^{2s}}^{(p-1)^2+1}}}{\displaystyle{(\eta\hspace{1pt}\eta_{p^n})^{p-1}}}}&\text{ if $n$ is even}
\end{cases}
\label{hpn}\end{equation}
and note that 
\begin{equation}
  \mathcal{H}_{2,\text{odd\hspace{1pt}$n$}}=\frac{\eta_2^2}{\eta}\times\frac{\eta_{2^{n-1}}^2}{\eta_{2^n}}, \ \mathcal{H}_{p,1}(z)=\eta\eta_p, \  \mathcal{H}_{p,3}(z)=\mathcal{H}_{p,2}(z)\mathcal{H}_{p,2}(pz).
\label{redH}\end{equation}
 \begin{theorem} For a prime $p$ and an $n\in\N$, the primitive 
cuspidal eta quotient $\mathcal{H}_{p,n}$ is irreducible except only for the cases in \eqref{redH}.
\label{prp1}\end{theorem}
\begin{co}
For every $k\in\N$, there exists a primitive and irreducible 
cuspidal eta quotient 
of weight $k/2$ and level $2^{2k}$, viz. $\mathcal{H}_{2,2k}$.
\label{snld}\end{co}
The above corollary implies that 
for all $k\in\N$, the $k$-th Mersmann level $\cl{M}_k$ is divisible by $2^{2k}$. More generally, we have the following result.
\begin{co}\label{12:42PM} For all $k\in\N$, the $k$-th Mersmann level $\cl{M}_k$ is divisible by
    $$\prod_{\substack{p \ \text{ prime}\\p\leq\sqrt{k-\frac34}+\frac32}}p^{\varepsilon_{p,k}},$$
where $$\varepsilon_{p,k}=\begin{cases}
2k&\text{if \ $p=2$}\\
\mathbf{I}_{\N}\Big(\frac{2k+p-3}{(2p-3)(p-1)}+1  \Big)&\text{if \ $p>2$ and  $k\notin2\N$}\\
\mathbf{I}_{\N_{>3}}\Big(\frac{k-2}{(p-1)(p-2)}+1  \Big)&\text{if \ $p>2$ and  $k\in2\N$.}\\
\end{cases}$$
 Here $$\mathbf{I}_{\N}(n)=\begin{cases}
n&\text{if $n\in\N$}\\0&\text{otherwise}
 \end{cases}
 \text{ \ and \ \ }
 \mathbf{I}_{\N_{>3}}(n)=\begin{cases}
n&\text{if $n\in\N$ and $n>3$}\\0&\text{otherwise.}
 \end{cases}
 $$
\end{co}

\
 
  In the following, by a \emph{rescaling} of an eta quotient $f$ of level $N$, we mean the image of $f$
under the rescaling by a positive integer and by \emph{image of~$f$ under an Atkin-Lehner involution}, 
we mean the image of $f$
under an Atkin-Lehner involution of 
some level $M\in N\Z$.

\begin{co}
Except for the cases in \eqref{redH}, the following assertions hold:

\noindent
 $(a)$ All the rescalings of $\mathcal{H}_{p,n}$ are irreducible.\\ 
  $(b)$ All the images of $\mathcal{H}_{p,n}$ under Atkin-Lehner involutions are irreducible.\\
$(c)$ All the rescalings of the images of $\mathcal{H}_{p,n}$ under Atkin-Lehner involutions are irreducible.\\
$(d)$ All the images under Atkin-Lehner involutions of the rescalings\footnote{Atkin-Lehner involutions and rescalings do not commute with each other in general.} 
 of $\mathcal{H}_{p,n}$ are irreducible.\\
$(e)$ Under every arbitray finite sequence of compositions of Atkin-Lehner involutions and rescalings, the image of $\mathcal{H}_{p,n}$ 
is irreducbile.
\label{cor2}\end{co}

\section{Preliminaries}
\label{26.08.2015.1}
For $N\in\N$, by $\bb{D}_N$ we denote the set of divisors of $N$.
For 
$X\in\Z^{\bb{D}_N}$, we define the eta quotient $\e^X$
   by
   \begin{equation}\label{3Jan15.1}
    \e^X:=\displaystyle{\prod_{\ \hspace*{4pt}{d\in\bb{D}_N}}\eta_d^{X_d}},
\end{equation}
where $X_d$ is the 
value of $X$ 
at $d\in\bb{D}_N$ whereas $\e_d$ denotes the rescaling of $\e$ by $d$.
Clearly, the level of $\e^X$ divides $N$. In other words, $\e^X$ transforms like a modular form on $\GN$. 
Recall that an eta quotient $f$ on $\GN$ is holomorphic
if it 
does not have any poles at the cusps of $\GN$.
We define the \emph{\hypertarget{om}{order map}}
$\cl{O}_N:\Z^{\bb{D}_N}\rightarrow\frac1{24}\Z^{\bb{D}_N}$ of level $N$ as the map which sends $X\in\Z^{\bb{D}_N}$ to the ordered set of
orders of vanishing of the eta quotient $\e^X$ at the cusps $\{1/t\}_{\substack{\ \\ t\in\bb{D}_N}}$ of $\GN$ (see~\eqref{27.04.2015.1}). We define the
\emph
{order matrix} $A_N\in\Z^{\bb{D}_N\times\bb{D}_N}$ of level $N$ by 
\begin{equation}
 A_N(t,d):=24\cdot\ord_{1/t}(\e_d\hspace{1pt};\GN) 
\label{27.04.2015}\end{equation}
for all $t,d\in\bb{D}_N$.
 For example, for a prime power $p^n$, 
we have
\begin{equation}
A_{p^n}=\begin{pmatrix}
 \vspace{5.8pt}p^n &p^{n-1} &p^{n-2} 
&\cdots  &p &1\\
\vspace{5.8pt} p^{n-2} &p^{n-1} &p^{n-2} 
&\cdots  &p &1\\
 \vspace{5.8pt}p^{n-4} &p^{n-3} &p^{n-2} 
&\cdots  &p &1\\
\vspace{5.8pt} \vdots &\vdots &\vdots 
&\cdots &\vdots &\vdots\\
\vspace{5.8pt} 1 &p &p^2 
&\cdots &p^{n-1} &p^{n-2}\\
 1 &p &p^2 
&\cdots  &p^{n-1} &p^n
\end{pmatrix}.
 \label{23July}
\end{equation}
More precisely, for all integers $i,j\in[0,n]$, 
\begin{equation}
 A_{p^n}(p^i,p^j)=p^{n-\min\{i,n-i\}-|i-j|}.
\label{12:47PM}\end{equation}
By linearity of the order map, we have 
\begin{equation}
\cl{O}_N(X)=\frac1{24}\cdot A_NX\in\frac1{24}\ZD. 
\label{28.04}\end{equation}
For $r\in\N$, if $Y,Y'\in\Z^{\bb{D}_N^{\hspace{.5pt}r}}$ is 
such that $Y-Y'$ is
nonnegative at each element of $\bb{D}_N^{\hspace{.5pt}r}$, then
we write $Y\geq Y'$. 
In particular, for $X\in\Z^{\bb{D}_N}$, the eta quotient $\e^X$ is holomorphic (resp. cuspidal) if and only if $A_NX\geq0$ (resp. $A_NX>0$).

From $(\ref{27.04.2015})$ and $(\ref{27.04.2015.1})$, we note that $A_N(t,d)$ is multiplicative in $N, t$ and~$d$.
Hence, it follows that
\begin{equation}
 A_N=\bigotimes_{\substack{p^n\|N\\\text{$p$ prime}}}A_{p^n},
\label{13May}\end{equation}
where by \hspace{1.5pt}$\otimes$\hspace{.2pt},  we denote the Kronecker product of matrices.\footnote{Kronecker product of matrices is
not commutative. However, 
since any given ordering of the primes dividing $N$ induces a lexicographic ordering on $\bb{D}_N$ 
with which the entries of $A_N$ are indexed, 
Equation (\ref{13May}) makes sense for all possible 
orderings of the primes dividing $N$.} 

We define the operation $\odot:\N\times\N\to\N$ by
\begin{equation}
    d_1\odot d_2=\frac{\lcm(d_1,d_2)}{\gcd(d_1,d_2)}.
\label{11:42AM}\end{equation}
A divisor $d$ of $N$ is called unitary, written $d\|N$, if $\gcd(d, N/d)=1$.
By $\bb{U}_N$, we denote the set of unitary divisors
of $N$. It follows trivially that $\bb{U}_N$ is a boolean group under $\odot$ (i.e. each element is the inverse of itself) and 
the action of
$\bb{U}_N$ by $\odot$ induces an involution on $\bb{D}_N$.
From \eqref{12:47PM}, we obtain 
\begin{equation}
A_{p^n}(p^i,p^j)=A_{p^n}(p^{n-i},p^{n-j})=A_{p^n}(p^n\odot p ^i,p^n\odot p ^j).
\nonumber\end{equation}
It follows from the above equation and \eqref{13May} that
for all $u\in\bb{U}_N$ and for all $t,d\in\bb{D}_N$, we have
\begin{equation}
A_{N}(t,d)=A_{N}(u\odot t, u\odot d).
\nonumber\end{equation}
For $X\in\Z^{\bb{D}_N}$, we have $A_NX\geq0$ if and only if for each $d\in\bb{D}_N$, 
\begin{equation}
0\leq \sum_{t|N}A_N(t,d)X_d=\sum_{t|N}A_{N}(u\odot t, u\odot d)X_d.
\label{6:30AM}\end{equation}
For $u\in\bb{U}_N$,
the Atkin-Lehner in{\-}volution $\mathcal{AL}_{u,N}$ of the eta quotients on $\GN$
is defined by
\begin{equation}
\mathcal{AL}_{u,N}\Big{(}\prod_{d|N}\eta_d^{X_d}\Big{)}  \  := \  \prod_{d|N}\eta_{u\odot d}^{X_d}.
\label{6:22PM}\end{equation} 
The Fricke involution is a special case of Atkin-Lehner involutions, defined~by
\begin{equation}
    W_N:=\mathcal{AL}_{N,N}.
\nonumber\end{equation}
The Atkin-Lehner involutions of level $N$ form a boolean group $(\cl{G}_N,\circ)$, where the group law is induced by the group law of $\bb{U}_N$, viz. for $u, u'\in\bb{U}_N$, 
$$\mathcal{AL}_{u,N}\circ\mathcal{AL}_{u',N}:=\mathcal{AL}_{u\odot u',N}.$$
Let $\cl{C}_N$ denote the set of the cusp classes of $\GN$. The group action
\begin{equation}
    \cl{G}_N\times\cl{C}_N\longrightarrow\cl{C}_N
\label{11:18PM}\end{equation}
is induced by the action of $\bb{U}_N$ on $\bb{D}_N$, viz. for $u\in\bb{U}_N$, 
the action of 
$\mathcal{AL}_{u,N}$ on $\cl{C}_N$ is defined by the involution which maps 
the cusp class represented by $1/t$ to the one represented by $1/(u\odot t)$.

For $t\in\bb{D}_N$, $u\in\bb{U}_N$ and for each eta quotient $f$ on $\GN$, from
\eqref{27.04.2015},
\eqref{6:30AM} and \eqref{6:22PM},
we have 
\begin{equation}
    \ord_{1/t}(f\hspace{1pt};\GN)=\ord_{1/(u\odot t)}(\mathcal{AL}_{u,N}(f)\hspace{1pt};\GN).
\label{11:03AM}\end{equation}
It follows from \eqref{11:03AM} and \eqref{6:22PM}
that
\begin{lem}$(a)$ 
Holomorphy and cuspidality of eta quotients are preserved by
Atkin-Lehner involutions.
\newline$(b)$ Let $u$ be a unitary divisor of $N$. For eta quotients $f,g$ and $h$ on $\GN$, we have 
$f=gh$ if and only if $$\mathcal{AL}_{u,N}(f)=\mathcal{AL}_{u,N}(g)\times\mathcal{AL}_{u,N}(h).$$
\label{7:03PM}\end{lem}

It is easy to verify that for a prime power $p^n$, 
the matrix $A_{p^n}$ is invertible with the tridiagonal inverse: 
\begin{equation}
A_{p^n}^{-1}=
\frac{1}{p^{n-1}(p^2-1)}
\begin{pmatrix}
\hspace{6pt}p &\hspace*{-6pt}-p &  
& &  & \\
\vspace{5pt}\hspace*{-1pt}-1 & \hspace*{-2pt}p^2+1 &\hspace*{-3pt}-p^2 
& &\textnormal{\Huge 0} & \\
\vspace{5pt}&\hspace*{-6pt}-p & \hspace*{-4pt}p\cdot(p^2+1) &\hspace*{-2pt}-p^3 
 &  &  \\
 &  &\hspace*{-5pt}\ddots 
 &\hspace*{-6pt}\ddots & \ddots & \\
\vspace{5pt}\hspace{2pt} &\textnormal{\Huge 0}  &  
 &\hspace*{-7pt}-p^2 & \hspace*{-4pt}p^2+1 &\hspace*{-4pt}-1\hspace{2pt}\\
\hspace{2pt} &  & 
&  &\hspace*{-6pt}-p & p\hspace{2pt}
\end{pmatrix},
 \label{r1}\end{equation}
where 
for each positive integer $j<n$, 
the nonzero entries of the column  $A_{p^n}^{-1}(\un\hs,p^j)$ are the same as 
those of the column $A_{p^n}^{-1}(\un\hs,p)$ shifted down by $j-1$ entries and
multiplied with $p^{\min\{j-1,n-j-1\}}$. 
More precisely\footnote{In particular, only the row sum(s) of the middle row(s) of $A_{p^n}^{-1}$ are nonzero.},
\begin{align}
p^{n-1}\cdot(p^2-1)&
\cdot A_{p^n}^{-1}(p^i,p^j)\nonumber=\\&\begin{cases}
                        \ \hs\ p&\text{if $i=j=0$ or $i=j=n$}\\
                        -p^{\min\{j,n-j\}}&\text{if $|i-j|=1$}\\
                        \ \hs\ p^{\min\{j-1,n-j-1\}}\cdot(p^2+1)&\text{if $0<i=j<n$}\\
                        \ \hspace{1.2pt}\ 0&\text{otherwise.}
\end{cases}
\label{apninv}\end{align}
For general $N$, the invertibility of the matrix $A_N$ now 
follows by (\ref{13May}).
Hence, an eta quotient on $\GN$ is uniquely determined by its orders at the set of the cusps
$\{1/t\}_{\substack{\ \\ t\in\bb{D}_N}}$ of $\GN$. 
In particular, for distinct $X,X'\in\Z^{\bb{D}_N}$, 
\begin{equation}
\e^X\neq\e^{X'}.
\label{6:50PM}\end{equation}
Since all the entries of $A_N^{-1}$
are rational (see (\ref{13May}) and (\ref{r1})), for each $t\in\bb{D}_N$, there exists a smallest positive integer $m_{{t,N}}$ such that 
$m_{{t,N}}\cdot A_N^{-1}(\un\hs,t)$
has integer entries, where $A_N^{-1}(\un\hs,t)$ denotes the 
column of $A_N$ indexed by $t\in\bb{D}_N$.
It follows from \eqref{apninv} and \eqref{13May} that 
\begin{equation}
 m_{t,N}=\frac{N\hs\rad(\gcd(t,N/t))}{\gcd(t,N/t)\hs\rad(N)}\prod_{\substack{p|N\\p\text{ prime}}}(p^2-1),
\label{mtn}\end{equation}
where 
for $n\in\N$, 
$\rad(n)$ denotes the product of the distinct prime divisors of $n$.
Define the 
matrix $M_N\in\Z^{\bb{D}_N\times\bb{D}_N}$ by
\begin{equation}
M_N(d,t)=\begin{cases}m_{{t,N}}&\text{ if $d=t$}\\ 0&\text{ otherwise.}
         \end{cases}
\label{11:53AM}
\end{equation}
and define $B_N\in\Z^{\bb{D}_N\times\bb{D}_N}$ by
\begin{equation}
 B_N:=A_N^{-1}M_N.
\label{28.08.2015A}\end{equation}
That implies
\begin{equation}
B_N(\un\hs,t):=m_{{t,N}}\cdot A_N^{-1}(\un\hs,t)\hspace{4pt}\text{for all $t\in\bb{D}_N$,}
\label{10:39AM}\end{equation}
where $B_N(\un\hs,t)$  denotes the 
column of $B_N$ indexed by $t\in\bb{D}_N$.
It follows from \eqref{28.04} that
\begin{equation}
 \ord_{1/t'}\e^{B_N(\un\hs,\hs t)}=\begin{cases}{m_{t,N}}/{24}&\text{ if $t'=t$}\\
                                0&\text{ otherwise.}\\
                               \end{cases}
\label{ordbncols}\end{equation}
Since $m_{t,N}\neq0$ for all $t\in\bb{D}_N$, 
from \eqref{28.08.2015A}, we see that $B_N$ is invertible over~$\Q$. 
If $X\in\Z^{\bb{D}_N}$ is such that $\e^X$ is holomorphic and for $Y=B_N^{-1}X$,
if there exists $d\in\bb{D}_N$ such that $Y_d\geq1$, then it follows from 
\eqref{28.04}, \eqref{28.08.2015A}, \eqref{11:53AM} and \eqref{ordbncols}
that $\e^{B_N(\un\hs,\hs t)}$ is a factor of $\e^X$. Thus, we obtain:
\begin{lem}[Lemma~2 from \cite{B3}] For $N\in\N$, let $B_N\in\Z^{\bb{D}_N\times\bb{D}_N}$ be as defined in $(\ref{28.08.2015A})$. 
For $X\in\Z^{\bb{D}_N}$, if $\e^X$ is a holomorphic eta quotient which is not factorizable on $\GN$,
then either $X\in B_N\cdot[0,1)^{\bb{D}_N}$ or $X=B_N(\un\hs,t)$ for some $t\in\bb{D}_N$.
\label{b2}\end{lem}
\begin{co}
For $X\in\Z^{\bb{D}_N}$, if $\e^X$ is a generic holomorphic eta quotient which is not factorizable on $\GN$, then $X\in B_N\cdot[0,1)^{\bb{D}_N}$.
\label{9:58AM}\end{co}
We define the eta quotient 
$\mathcal{F}_N$ by
\begin{equation}
F_N:=\prod_{t\in\D}\e^{B_N(\un\hs,\hs t)}.
\label{5:25AM}\end{equation}
\begin{lem}[Lemma~1 from \cite{B3}]
For $N\in\N$, let $\cl{F}_N$ be as defined above. Then for $X\in\ZD$, both the eta quotients $\e^X$ and $\cl{F}_N/\e^X$ are holomorphic if and and only if $X\in B_N\cdot [0,1]^{\bb{D}_N}$.
\label{5:26AM}\end{lem}
\begin{co}For $N\in\N$ and $X\in\ZD$, both the eta quotients $\e^X$ and $\cl{F}_N/\e^X$ are cuspidal if and and only if $X\in B_N\cdot (0,1)^{\bb{D}_N}$.
\label{5:46AM}\end{co}
From the multiplicativity of $m_{t,N}$ in $N$ and $t$
and from the multiplicativity of $A_N^{-1}(d,t)$ in $N$, $d$ and $t$  (see (\ref{13May})),
it follows that $B_N(d,t)$ 
is also multiplicative in $N$, $d$ and $t$. 
That implies
\begin{equation}
 B_N=\bigotimes_{\substack{p^n\|N\\\text{$p$ prime}}}
 B_{p^n}.
\label{12:24AM}\end{equation}
In particular, for a prime $p$, 
from (\ref{28.08.2015A}) and 
(\ref{r1}), we have 
 \begin{equation}
 B_{p^n}=
\begin{pmatrix}
\hspace{6pt}p &\hspace*{-6pt}-p &  
& &  & \\
\vspace{5pt}\hspace*{-1pt}-1 & \hspace*{-2pt}p^2+1 &\hspace*{-3pt}-p 
& &\textnormal{\Huge 0} & \\
\vspace{5pt}&\hspace*{-6pt}-p & \hspace*{-4pt}p^2+1 &\hspace*{-2pt}-p 
 &  &  \\
 &  &\hspace*{-5pt}\ddots 
 &\hspace*{-6pt}\ddots & \ddots & \\
\vspace{5pt}\hspace{2pt} &\textnormal{\Huge 0}  &  
 &\hspace*{-7pt}-p & \hspace*{-4pt}p^2+1 &\hspace*{-4pt}-1\hspace{2pt}\\
\hspace{2pt} &  & 
&  &\hspace*{-6pt}-p & p\hspace{2pt}
\end{pmatrix},
\label{equ1}
\end{equation}
i.e. $B_{p^n}$ is a tridiagonal matrix, the nonzero entries in each of whose columns except the ones indexed by $1$ and $p^n$
are $-p,p^2+1$ and $-p$.


\

\section{Proofs of the results}

{\flushleft\textit{Proof of Theorem~\ref{sth}.}}
Follows from Corollary~\ref{9:58AM} and Lemma~\ref{00:36AM}.\qed

\

{\flushleft\textit{Proof of Corollary~\ref{17:02PM}.}} Let $\lambda_N$ denote the cardinality of $\Lambda_N$.
Then it follows from Definition~\ref{9:22AM} that if $N=\prod_{p|N}p^{n_p}$, then
\begin{equation}\lambda_N=\underset{p|N}{\prod}(n_p+1)-\underset{p^2|N}{\prod}(n_p-1),
\label{6:04AM}\end{equation}
where an empty product equals 1.
Let $r$ denote the number of distinct prime divisors of $N$. 
From \eqref{6:04AM}, we obtain 
\begin{equation}\lambda_N\leq\underset{p|N}{\prod}(n_p+1)-\underset{p|N}{\prod}(n_p-1)\leq 2\Big(\sum_{p|N}n_p\Big)^{r-1},
\label{00:12AM}\end{equation}
where the last inequality follows by 
induction on $r$.
Since 
$$\sum_{p|N}n_p\leq\log_2 N,$$
the claim follows from Theorem~\ref{sth} and \eqref{00:12AM}.
\hfill\qed

\ 
 
{\flushleft{\textbf{Aside:}}}  By $\mathrm{d}(n)$, we denote the number of divisors of $n$. Then \eqref{00:12AM} implies:
\begin{equation}
\mathrm{d}(N)\leq2(\log_2 N)^{r-1}+\underset{p|N}{\prod}(n_p-1)
\label{6:08PM}\end{equation}
\begin{co}If $N$ has a unitary prime divisor and if the number of distinct prime divisors of $N$ is $r$, then 
$$\mathrm{d}(N)\leq 2(\log_2 N)^{r-1}.$$
Moreover, if $N$ is odd, 
then we have $\mathrm{d}(N)\leq 2(\log N)^{r-1}.$
\end{co}
\begin{co}Let $N>1$ have $r$ distinct prime divisors,
let $q$ be the least prime 
such that $q^2|N$, let $N':=N/\rad(N)$ and let 
$\rho:=\rad(N')$.
Then 
$$d(N)\leq 2(\log_2 N)^{r-1}+(\log_{\rho}N')^{\log_q\rho}.$$
\end{co}
\begin{proof}
From \eqref{6:08PM}, it follows that $\mathrm{d}(N)-2(\log_2 N)^{r-1}$ is bounded above 
by the quantity on the left hand side of the inequality
\begin{equation}
\displaystyle{\prod_{p|N'}(n_p-1)^{\log_q p}}\leq\left(\frac{\displaystyle{\sum_{p|N'}(n_p-1)\log_q p}}{\displaystyle{\sum_{p|N'}\log_q p}}\right)^{\log_{q}\rad(N')},
\nonumber\end{equation}
which follows from the weighted AM-GM inequality\footnote{See 
also Theorem~1 in \cite{B8}}. 
\end{proof}

\

{\flushleft\textit{Proof of Corollary~\ref{ppi}.}}
Let $f$ be a generic irreducible holomorphic eta quotient of a prime power level $p^n$.
Since $\{1,1/p^n\}$ represents
a special set of cusp classes  
of $\Gamma_0(p^n)$,
the claim follows 
from Theorem~\ref{sth} and the fact that $1$ and~$0$ (resp. $1/p^n$ and $\infty$) represent the same cusp
of~$\Gamma_0(p^n)$.\hfill\qed

\ 

{\flushleft\textit{Proof of Corollary~\ref{a-l}.}}
If $p\|N$, then
a holomorphic eta quotient $f$ which is invariant under 
$\mathcal{AL}_{p,N}$
is 
reducible.
More precisely, it follows from \eqref{6:22PM} that
$f(z)=g(z)g(pz)$ for some holomorphic eta quotient $g$ on $\Gamma_0(N/p)$. 
Hence, if there exists an irreducible holomorphic eta quotient which is
an Atkin-Lehner eigenform of level $N$, then $N$ 
has no 
unitary prime divisor.

Let $N=\prod_{p|N}p^{n_p}$ have no unitary prime divisors, let
$\cl{G}_N$ denote the group of Atkin-Lehner involutions of level $N$ and let 
$\cl{C}_N$ denote the cusp classes of $\GN$. 
The action \eqref{11:18PM} of $\cl{G}_N$ on $\cl{C}_N$ partitions the latter into
\begin{equation}
    \prod_{p|N}\Big\lceil\frac{n_p+1}{2}\Big\rceil
\label{12:23PM}\end{equation}
orbits. 
Since $N$ has no unitary prime divisor, there exists
a special set of cusp classes $\cl{S}_N$ of $\GN$ defined  the sequence $\{\alpha_p\}_{p|N}$ 
such that $\alpha_p=0$ for all $p$. 
It follows from \eqref{11:18PM}  that
for each unitary divisor $u$ of $N$ and for all $p|N$, the Atkin-Lehner involution $\mathcal{AL}_{u,N}$ acts on $\cl{S}_N$ by interchanging $\alpha_p$ and $n_p$ if $p|u$ and keeping $\alpha_p$ and $n_p$
unchanged otherwise. Thus, $\cl{S}_N$ is closed under the action of $\cl{G}_N$ on $\cl{C}_N$. In other words, \eqref{11:18PM} defines actions of $\cl{G}_N$ on both 
the sets $\cl{S}_N$ and $\cl{C}_N\smallsetminus\cl{S}_N$.
Since under the action of $\cl{G}_N$, the set $\cl{C}_N\smallsetminus\cl{S}_N$ is 
partitioned into 
\begin{equation}
    \prod_{p|N}\Big\lceil\frac{n_p-1}{2}\Big\rceil
\label{12:56PM}\end{equation}
orbits, it follows from \eqref{12:23PM} that the action of $\cl{G}_N$ partitions $\cl{S}_N$ 
into 
\begin{equation}
    \prod_{p|N}\Big\lceil\frac{n_p+1}{2}\Big\rceil-
    \prod_{p|N}\Big\lceil\frac{n_p-1}{2}\Big\rceil
\label{12:59PM}\end{equation}
orbits.
Theorem~\ref{sth} implies that a generic irreducible holomorphic eta quotient $f$ on $\GN$ is uniquely identified by its orders of vanishing at the elements in $\cl{S}_N$.  
If $f$ is an Atkin-Lehner eigenform of level $N$, then it remains invariant under 
all the Atkin-Lehner involutions in $\cl{G}_N$. 
Now, \eqref{11:03AM} implies that the orders of vanishing of $f$ remains constant at every $\cl{G}_N$-orbit in $\cl{S}_N$.
In particular, the orders of vanishing of $f$ only at the \mbox{$\cl{G}_N$-orbit} representatives in $\cl{S}_N$ identifies $f$ uniquely.  
\eqref{12:59PM}. \hfill\qed

\ 

{\flushleft\textit{Proof of Corollary~\ref{fi}.}}
Let $f$ be a generic irreducible holomorphic eta quotient and a Fricke eigenform of a prime power level $p^n$. 
The claim then follows 
from Corollary~\ref{ppi}, because while $f$
remains invariant under the Fricke involution $W_{p^n}$, 
the action of $W_{p^n}$ on 
the cusp classes of $\Gamma_0(p^n)$  (see \eqref{11:18PM} and \eqref{11:03AM})
interchanges the cusp classes represented by 1 and $1/p^n$, which 
are the same as the cusp classes of 0 and $\infty$, respectively.\hfill\qed 

\

{\flushleft\textit{Proof of Theorem~\ref{prp1}.}} Follows from \S\ref{17.6Sept} and \S\ref{12:40pm}.\qed

\ 

{\flushleft\textit{Proof of Corollary~\ref{snld}.}}
From \eqref{hpn}, 
we see that the weight of 
$\mathcal{H}_{p,n}$ is\footnote{See Conjecture 2.($b$) in \cite{B3}.}
\begin{equation}
\begin{cases}
\displaystyle{\frac{n-1}2(p-1)^2-\frac{n-1}{2}(p-1)+1}&\text{ if $n$ is odd}\\ \\
\displaystyle{\frac{n-1}2(p-1)^2-\frac{n}{4}(p-1)+\frac12}&\text{ if $n$ is even.}
\end{cases}
\label{2:35AM}\end{equation} 
Hence, for even $p$ and $n$, the claim follows from Theorem~\ref{prp1}.\qed 

\

{\flushleft\textit{Proof of Corollary~\ref{12:42PM}.}} Follows from Theorem~\ref{prp1} and \eqref{2:35AM}.\qed

\ 

{\flushleft\textit{Proof of Corollary~\ref{cor2}.}}
 Corollary~5 in \cite{B5} states that each rescaling and every Atkin-Lehner involution
of an irreducible holomorphic eta quotient of a prime power level is irreducible. Hence, we obtain Corollary~\ref{cor2}.$(a)$ and $(b)$ 
as immediate consequences of 
Theorem~\ref{prp1}.

Since $\mathcal{H}_{p,n}$ is a Fricke eigenform
of level $p^n$, for all $M\in p^n\Z$, every Atkin-Lehner involution of $\mathcal{H}_{p,n}$ on $\Gamma_0(M)$ is just a rescaling of $\mathcal{H}_{p,n}$.
Since the composition of finitely many rescalings of $\mathcal{H}_{p,n}$ is only another rescaling of it,
Corollary~\ref{cor2}$.(b),(c),(d)$ and $(e)$ follow from 
Corollary~\ref{cor2}$(a)$\footnote{For an 
arbitrary irreducible holomorphic eta quotient of a prime power level, 
the analog
of Corollary~\ref{cor2}.$(e)$ also follows
in the 
same way as
the proof of Corollary~5 in \cite{B5}.}.
\hfill\qed

\section{The key lemma}
\label{5:19AM}
 
\begin{lem}Let $\cl{S}_N$ be a special set of cusps of $\GN$, let $B_N\in\Z^{\D\times\D}$ be as defined in \eqref{28.08.2015A} and let $$X\in B_N\cdot[0,1)^{\bb{D}_N}\cap\Z^{\bb{D}_N}$$ be such that 
$\e^X$ is a holomorphic eta quotient on $\GN$. Then $\e^X$ is uniquely identified by its 
orders of vanishing at the cusp classes in $\cl{S}_N$.
\label{00:36AM}\end{lem}
\begin{proof}

From \eqref{12:24AM}, we have
\begin{equation}
 B_N=\bigotimes_{\substack{p^n\|N\\\text{$p$ prime}}}B_{p^n}.
\end{equation}
By elementary row and column operations\footnote{i.e. by swapping rows (resp. columns) and subtracting a multiple of a row (resp. a column) from another.},
each $B_{p^n}$ can be reduced
to a diagonal matrix\footnote{i.e. the Smith Normal Form of $B_{p^n}$ (see~\cite{a}).} $D_{p^n}$ with integer entries.
In other words,
there exist $U_{p^n}, V_{p^n}\in\operatorname{GL}_{n+1}(\bb{Z})$ 
such that $D_{p^n}=U_{p^n} B_{p^n} V_{p^n}$. In \S\ref{mi}, we provide 
the matrices $U_{p^n}$ and $V_{p^n}$ explicitly\footnote{
Up to 
rearrangement of the
diagonal 
entries 
of $D_{p^n}$, which 
is equivalent to
permutation of the 
rows of $U_{p^n}$ and the columns of $V_{p^n}$ accordingly.}
such that
\begin{equation}
D_{p^n}=\operatorname{diag}(1,1,\hdots,1,p^{n-1},p^{n-1}(p^2-1)).
\label{6:48AM}\end{equation}
Define 
\begin{equation}
U_N:=\bigotimes_{\substack{p^n\|N\\\text{$p$ prime}}}U_{p^n}, \ \ V_N:=\bigotimes_{\substack{p^n\|N\\\text{$p$ prime}}}V_{p^n} \text{ \ and \ } D_N:=\bigotimes_{\substack{p^n\|N\\\text{$p$ prime}}}D_{p^n}.
\label{10:20AM}\end{equation}
Then $U_N,V_N\in\operatorname{GL}_{\mathrm{d}(N)}(\Z)$ and 
\begin{equation}D_N=U_NB_NV_N.\label{9:49AM}\end{equation}
In particular, since $U_{N}$ is invertible over $\Z$,
left multiplication by $U_{N}^{-1}$ induces an automorphism of $\Z^{\bb{D}_{N}}$, i.e.
$U_{N}^{-1}\Z^{\bb{D}_{N}}=\Z^{\bb{D}_{N}}$.
It follows from \eqref{9:49AM} that 
\begin{equation}
    B^{-1}_N\Z^{\bb{D}_N}= B^{-1}_NU_N^{-1}\Z^{\bb{D}_N}=V_ND^{-1}_N\Z^{\bb{D}_N}.
\label{9:59AM}\end{equation}
Let $X\in B_N\cdot[0,1)^{\bb{D}_N}\cap\Z^{\bb{D}_N}$ be such that $\e^X$ is 
holomorphic.
Then 
\begin{equation}
Y:=B_N^{-1}X\in[0,1)^{\bb{D}_N}\cap B^{-1}_N\Z^{\bb{D}_N}=[0,1)^{\bb{D}_N}\cap V_ND^{-1}_N\Z^{\bb{D}_N},
\label{2:33AM}\end{equation}
where the last equality follows from \eqref{9:59AM},
Since $[0,1)^{\bb{D}_N}\cap V_ND^{-1}_N\Z^{\bb{D}_N}$ is a complete set of the coset representatives of $\Z^{\bb{D}_N}$ in $V_ND^{-1}_N\Z^{\bb{D}_N}$, there is natural a bijection between 
$[0,1)^{\bb{D}_N}\cap V_ND^{-1}_N\Z^{\bb{D}_N}$
and $(V_ND^{-1}_N\Z^{\bb{D}_N})/\Z^{\bb{D}_N}$. In particular, since $D_N\in\Z^{\bb{D}_N\times\bb{D}_N}$ is a diagonal matrix, we have
\begin{equation}
    Y\equiv \sum_{d|N}\frac{\ell_d}{D_N(d,d)}V_N(\un\hs,d)\pmod{\Z^{\bb{D}_N}},
\label{10:01AM}\end{equation}
where $\ell_d\in\Z$ with $0\leq\ell_d< D_N(d,d)$ and 
$V_N(\un\hs,d)$ denotes the column of $V_N$ indexed by $d\in\D$. That implies, if $d\in\D$
is such that $D_N(d,d)=1$, then $\ell_d=0$. 
Let 
\begin{equation}
\Lambda_N':=\{d\in\D \ : \ D_N(d,d)>1\}
\label{7:58PM}\end{equation}
and let $C_{\Lambda_N'}\in\Z^{\D\times\Lambda_N'}$ be defined by 
$C_{\Lambda_N'}(t,d)=V_N(t,d)$ for all $t\in\D$ and $d\in\Lambda_N'$. Then
\eqref{10:01AM} is equivalent to
\begin{equation}
Y\equiv C_{\Lambda_N'}\widetilde{\ell}\pmod{\ZD},
\label{10:03AM}
\end{equation}
where
$\widetilde{\ell}\in[0,1)^{\Lambda_N'}$ be defined by $\widetilde{\ell}_d:=\ell_d/D_N(d,d)$ for all $d\in\Lambda_N'$. Also, for $\Lambda_N$ as in $\eqref{12:00PM}$, let $\widetilde{Y}\in[0,1)^{\Lambda_N}$ be defined by \begin{equation}
    \widetilde{Y}_t:=Y_t \ \text{ for all $t\in\Lambda_N$.}
\label{2:29AM}\end{equation}  
Let $\lambda_N:=|\Lambda_N|$.
It follows from \eqref{6:48AM}, \eqref{10:20AM} and \eqref{7:58PM} that
the cardinality of $\Lambda_N'$ is given by the quantity on the right hand side \eqref{6:04AM}. Thus, we obtain
$$|\Lambda_N'|=|\Lambda_N|=\lambda_N.$$
Let $E_N\in\Z^{\Lambda_N\times\Lambda_N'}$ be defined by
\begin{equation}
    E_N(t,d)=V_N(t,d) \  \text{ for all $t\in\Lambda_N$ and $d\in\Lambda_N'$.}
\label{1:12AM}\end{equation}
It follows from \eqref{10:03AM} that 
\begin{equation}
    \widetilde{Y}\equiv E_N\widetilde{\ell}\pmod{\Z^{\Lambda_N}}.
\label{1:38AM}\end{equation}
In Lemma~\ref{1:36AM}, we show that $E_N$ is invertible over $\Z$. So, $E_N^{-1}\Z^{\Lambda_N}=\Z^{\Lambda_N}$. Hence, \eqref{1:38AM} is equivalent to 
\begin{equation}
    \widetilde{\ell}\equiv E_N^{-1}\widetilde{Y}\pmod{\Z^{\Lambda_N}}.
\label{2:30AM}\end{equation}
Since $Y=B_N^{-1}X$, from  \eqref{28.08.2015A}, \eqref{11:53AM}, \eqref{28.04} and \eqref{2:29AM},
it follows that the entries of $\widetilde{Y}_t$ are the normalized orders of the 
holomorphic eta quotient $\e^X$ at the special set of cusp classes of $\GN$ which are represented by the reciprocals of the elements of $\Lambda_N$ (see \eqref{12:00PM}).
Hence, the claim follows from \eqref{2:30AM}, \eqref{10:03AM} and \eqref{2:33AM}.
\end{proof}

\begin{lem}Let $E_N$ be as in \eqref{1:12AM}. 
Then $E_N\in\operatorname{GL}_{\lambda_N}(\Z).$
\label{1:36AM}\end{lem}
\begin{proof}
For a prime $p$ and an $n\in\N$, 
$V_{p^n}\in\operatorname{GL}_{n+1}(\Z)$ is given by (see \S\ref{mi}):
\begin{equation}
V_{p^n}=
\begin{cases}\begin{pmatrix}1&p\\0&1\end{pmatrix}&\text{if $n=1$}\\
\\
\begin{pmatrix}1&
    0\hspace{2pt}\cdots\hspace{2pt}0&1&p^n\\
    0& 0\hspace{2pt}\cdots\hspace{2pt}0
    &1&p^{n-2}\\\begin{matrix}0\\\vdots\\0\end{matrix}
&L_{p^n}&\begin{matrix}\ast\\\vdots\\\ast\end{matrix}&\begin{matrix}\ast\\\vdots\\ \ast\end{matrix}\\
 0&0\hspace{2pt}\cdots\hspace{2pt}0
    &0&1
    \end{pmatrix}
    &\text{otherwise,}
    \end{cases}
\label{5:55AM}\end{equation}
where $L_{p^n}\in\operatorname{GL}_{n-2}(\Z)$ 
is defined by
\begin{equation}L_{p^n}(i,j)=\begin{cases}
-p^{i-j}&\text{if  $i\geq j$,}\\  
0&\text{otherwise.}    
\end{cases}
\label{10:47AM}\end{equation}
From \eqref{12:00PM}, \eqref{7:58PM} and \eqref{6:48AM}, 
it follows that $\Lambda_p=\Lambda_p'=\{p\}$. Hence, from \eqref{5:55AM} 
we obtain that
$E_p=[V_p(p,p)]=[1]\in\operatorname{GL}_1(\Z).$

For $n>1$, it follows from \eqref{12:00PM} that 
$\Lambda_{p^n}=\{1,p^n\}\text{ \ or \ }\{p,p^n\}$
while \eqref{7:58PM} and \eqref{6:48AM} imply that $\Lambda_{p^n}'=\{p^{n-1},p^n\}.$
So, $\eqref{5:55AM}$ implies that
\begin{equation}
    E_{p^n}=\begin{pmatrix}1&p^n\\0&1\end{pmatrix} \text{ \ or \ }\begin{pmatrix}1&p^{n-2}\\0&1\end{pmatrix},
\end{equation}
both of which belong to $\operatorname{GL}_2(\Z)$ for all $n>1$.
Thus, the claim holds for all prime powers. 

Now, it suffices to show that if the claim holds for an $M\in\N$, then it also holds for $p^nM$ for all $n\in\N$, where $p$ is a prime such that $p\nmid M$. 

From $\eqref{10:20AM}$, we have $V_{p^nM}=V_{p^n}\otimes V_M$ and $D_{p^nM}=D_{p^n}\otimes D_M$. So, \eqref{6:48AM} implies that the diagonal\footnote{hence, a block diagonal matrix} matrix $D_{p^nM}\in\Z^{\mathbb{D}_{p^nM}\times\mathbb{D}_{p^nM}}$ is given by
\begin{equation}D_{p^nM}=\mathrm{diag}(D_M,D_M,\ldots,D_M,p^{n-1}D_M, p^{n-1}(p^2-1)D_M)
\label{8:21AM}\end{equation}
and from \eqref{5:55AM}, it follows that $V_{p^nM}\in\operatorname{GL}_{\mathrm{d}(p^nM)}(\Z)$ is of the form
\begin{equation}
V_{p^nM}=
\begin{cases}\begin{pmatrix}V_M&{\huge\ast}\\0&V_M\end{pmatrix}
&\text{if $n=1$}\\
\\
\begin{pmatrix}V_M&
    0&V_M&{\huge\ast}\\
    0& 0
    &V_M&{\huge\ast}\\
0&L_{p^n}\otimes V_M&{\huge\ast}&{\huge\ast}\\
 0&0
    &0&V_M
    \end{pmatrix}
    &\text{otherwise.}
    \end{cases}
\label{1:37PM}\end{equation}
It follows from \eqref{12:00PM}, \eqref{7:58PM} and \eqref{8:21AM} and \eqref{1:37PM} that  
$E_{p^nM}\in\Z^{\Lambda_{\Lambda_{p^nM}}\times\Lambda'_{\lambda_{p^nM}}}$ is of the form
\begin{equation}
E_{p^nM}=\begin{cases}
\begin{pmatrix}E_M&{\huge \ast}\\0&V_M\end{pmatrix}
&\text{if $n=1$}\\ \\
\begin{pmatrix}
    E_M&0&R_{\Lambda_M}&{\huge \ast}\\
    0&0&V_M&{\huge \ast}\\
    0&L_{p^n}\otimes E_M&{\huge \ast}&{\huge \ast}\\
    0&0&0&V_M
\end{pmatrix}
&\text{if $n>1$ and $\alpha_p=1$}\\ \\
\begin{pmatrix}
        C_{\Lambda_M'}&0&V_M&{\huge \ast}\\
    0&0&R_{\Lambda_M}&{\huge \ast}\\
    0&L_{p^n}\otimes E_M&{\huge \ast}&{\huge \ast}\\
    0&0&0&V_M
\end{pmatrix}
&\text{if $n>1$ and  $\alpha_p=0$}
\end{cases}
\label{10:35AM}\end{equation}
where 
$C_{\Lambda_M'}$ is the submatrix of $V_M\in\Z^{\mathbb{D}_M\times\mathbb{D}_M}$, 
composed of those columns of $V_M$ which are indexed by  $\Lambda_M'$ and $R_{\Lambda_M}$ is the submatrix of $V_M$, 
composed of those rows of $V_M$ which are indexed by  $\Lambda_M$. 

We note that $E_{pM}$ is already in a block upper triangular form, whereas
for $n>1$, after a set of column operations interchanging the second and third column
blocks 
of $E_{p^nM}$ in \eqref{10:35AM}, we also obtain a block upper triangular matrix with four diagonal blocks if $\alpha_p=1$, viz. $E_M, V_M,L_{p^n}\otimes E_M$ and $V_M$ 
 and three diagonal blocks if $\alpha_p=0$, viz. 
 \begin{equation}
 \begin{pmatrix}
        C_{\Lambda_M'}&V_M\\0&R_{\Lambda_M}\end{pmatrix}, \ L_{p^n}\otimes E_M \text{ \ and \ } V_M.
        \label{7:19AM}\end{equation}
Since by induction hypothesis, $E_M\in\operatorname{GL}_{\lambda_M}(\Z)$ and since
$L_{p^n}\in\operatorname{GL}_{n-2}(\Z)$ (see~\eqref{10:47AM}), it follows that 
$$L_{p^n}\otimes E_M\in\operatorname{GL}_{(n-2)\lambda_M}(\Z).$$
Since the determinant of a block upper triangular matrix is 
the product of the determinants of its diagonal blocks\footnote{See Lemma~1.25 in \cite{xz}.} 
and since $E_M,V_M$ and $L_{p^n}\otimes E_M$ are invertible over $\Z$,
the claim follows in both the cases where $n=1$ and $n>1$ with $\alpha_p=1$.

In the remaining case, where $n>1$ and $\alpha_p=0$, the diagonal blocks
after a set of column operations on $E_{p^nM}$, are given in
\eqref{7:19AM}. So, invertibility of $L_{p^n}\otimes\Z$ and $V_M$ over $\Z$ implies that
$\det E_{p^nM}=\pm\det Z_M$, where 
\begin{equation} Z_{M}:=\begin{pmatrix}
        C_{\Lambda_M'}&V_M\\0&R_{\Lambda_M}\end{pmatrix}.
\label{4:21PM}\end{equation}
Since the Schur complement\footnote{See \S1.11 in \cite{xz}.}
of $V_M$ is $-E_M$ in the matrix obtained
after interchanging  either the row blocks or the column blocks of $Z_M$ in \eqref{4:21PM},
it follows that $$\det Z_M=\pm(\det E_M).(\det V_M)=\pm1.$$ Thus, we obtain 
that 
$E_{p^nM}\in\operatorname{GL}_{\lambda_{p^nM}}(\Z).$ Hence, the claim follows.
\end{proof}

\section{Cuspidality of $\mathcal{H}_{p,n}$}
\label{17.6Sept}

We distinguish the cases, $p=2$ and $p\geq 3$.
\vspace{5pt}

{\flushleft\textbf{Case I.} ($p=2$)}
\vspace{4pt}

From  \eqref{hpn}, we obtain 
\begin{equation}
     \mathcal{H}_{2,n}=\begin{cases}\displaystyle{\frac{\eta_2^2}{\eta}\times\frac{\eta_{2^{n-1}}^2}{\eta_{2^n}}}&\text{ if $n$ is odd}\\ \\
     \displaystyle{\frac{\e_2^2\e_{2^{n-1}}
    }{
    \e\e_{2^n}
    }\prod_{s=1}^{n/2-1}\frac{\e_{2^{2s}}^2}{\e_{2^{2s-1}}}}&\text{ otherwise.}\\
     \end{cases}
\label{10:53PM}\end{equation}
It follows from \eqref{27.04.2015.1} that for all integers $0\leq k\leq n$ and $0\leq l<n$, 
\begin{equation}
 \ord_{1/{2^k}}\Big(\frac{\e_{2^{l+1}}^2}{\e_{2^{l}}};\Gm_0(2^n)\Big)\geq0 \  \text{ \ and \ } 
  \ord_{1/{2^k}}\Big(\frac{\e_{2^{n-1}}^2}{\e_{2^{n}}};\Gm_0(2^n)\Big)\geq0,
\label{23:13PM}\end{equation}
where equality holds if $k\leq l<n$ and $k=n$, respectively. Also, \eqref{27.04.2015.1} implies that for all integers $0\leq m<n$, we have
\begin{equation}
 \ord_{1/{2^m}}\Big(\frac{\e_{2^{n-1}}}{\e_{2^{n}}};\Gm_0(2^n)\Big)>0.
  \label{23:10PM}\end{equation}
From \eqref{10:53PM} and \eqref{23:13PM}, it follows that $\cl{H}_{2,n}$ is cuspidal
if $n$ is odd. 

If $n$ is even, then it follows from \eqref{10:53PM}, \eqref{23:13PM} and \eqref{23:10PM} that for all integers $0\leq m<n$, we have
\begin{equation}
    \ord_{1/{2^m}}\big(\cl{H}_{2,n};\Gm_0(2^n)\big)>0.
\label{11:23PM}\end{equation}
Since $\cl{H}_{2,n}$ is invariant under the Fricke involution of level $2^n$ and since the action of the Fricke involution on the cusp classes of $\Gm_0(2^n)$ interchanges the cusp classes represented by $1/2^n$ and $1$, it follows from \eqref{11:03AM} and \eqref{11:23PM} that 
\begin{equation}
\ord_{1/{2^n}}\big(\cl{H}_{2,n};\Gm_0(2^n)\big)=\ord_{1}\big(\cl{H}_{2,n};\Gm_0(2^n)\big)>0.
\end{equation}
Hence, $\cl{H}_{2,n}$ is also cuspidal 
if $n$ is even. 

\

{\flushleft\textbf{Case II.} ($p\geq3$)}

\

It suffices\footnote{Since $\mathcal{H}_{p,1}=\e\e_p$ is cuspidal for all $p$ 
and the cuspidality of $\mathcal{H}_{p,2}$ follows from \eqref{27.04.2015.1}.}
to consider $\mathcal{H}_{p,n}$ with $n\geq3$.
Let
$X\in\Z^{D_{p^n}}$ be such that
 $\e^{X}=\mathcal{H}_{p,n}$.
It follows from \eqref{28.04}
that $\mathcal{H}_{p,n}$ is cuspidal if and only if each entry 
in $A_{p^n}X\in\Z^{D_{p^n}}$ is positive.
First, we show that each entry of $A_{p^n}X$ except 
possibly the couple of entries indexed by $1$ and $p^n\in\bb{D}_{p^n}$ is positive.

For $n\geq3$ and for every integer
$j\in[0,n]$,
the exponent of $\e_{p^j}$ in $\mathcal{H}_{p,n}$ 
is larger than or equal to $(p-1)^2-p$
except the exponents of $\e$ and $\e_{p^n}$,
both of which are equal to $-(p-1)$. 
Since each entry of $A_{p^n}$ is positive, 
since the sum of the entries of $A_{p^n}$ having column indices $p^j$ with $j\in[1,n-1]$
in each row $A_{p^n}$ is larger than or equal to $p$ times the sum of the entries with column indices $1$ and $p^n$ in
that row (see~\eqref{23July} and \eqref{12:47PM})
and since $(p-1)^2-p>1-1/p $ for each odd prime~$p$, it follows that each entry of $A_{p^n}X$ is
positive,
except possibly 
the couple of entries indexed by $1$ and $p^n$.

Thus, 
it suffices to show
the positivity
the orders of vanishing of $\mathcal{H}_{p,n}$ at the cusps 1 and $1/p^n$ of $\Gm_0(p^n)$. 
Since the Fricke involution on $\Gm_0(p^n)$ interchanges these two cusps and since $\mathcal{H}_{p,n}$ is invariant under this Fricke involution,
its orders of vanishing at these two cusps of $\Gm_0(p^n)$ are the same
(see \eqref{11:03AM}).
Since $1/p^n$ and $\infty$ represent the same cusp of $\Gm_0(p^n)$ (see \eqref{minset}),
in order to show that $\mathcal{H}_{p,n}$ is cuspidal,
it suffices only to show that
$\ord_{\infty}\mathcal{H}_{p,n}$ is 
positive\footnote{Since the width of the cusp $\infty$ of $\Gm_0(n)$ is 1 for all $n\in\N$ (see~\cite{ds, ra}), for an eta quotient 
$f$ of level $N$ and for all $M\in N\Z$ we have $\ord_{\infty}(f;\Gm_0(N))=\ord_{\infty}(f;\Gm_0(M))$, which we simply
denote by $\ord_{\infty}f$.}. 
From \eqref{hpn}, we obtain
\begin{equation}
\mathcal{H}_{p,n+2}/\mathcal{H}_{p,n}=\begin{cases}
\dfrac{\e_{p^n}^{(p-1)^2+p}\e_{p^{n+1}}^{(p-1)^2}}{\e_{p^{n-1}}^p\e_{p^{n+2}}^{p-1}}&\text{ if $n$ is even}\\ \\
\dfrac{\e_{p^n}^{(p-1)^2}\e_{p^{n+1}}^{(p-1)^2+1}}{\e_{p^{n-1}}^p\e_{p^{n+2}}^{p-1}}&\text{ if $n$ is odd.}
\end{cases}
\nonumber\end{equation}
In particular, that implies $\ord_{\infty}(\mathcal{H}_{p,n+2}/\mathcal{H}_{p,n})=0$ for all $n\in\N$. It follows by induction on $n$ 
that 
\begin{equation}
\ord_{\infty}\mathcal{H}_{p,n}=\begin{cases}\dfrac{1}{24} &\text{ if $n$ is even}\\
\\
                                  \dfrac{p+1}{24} &\text{ if $n$ is odd.}
                                 \end{cases}
\label{ordinfty}\end{equation}
Hence, $\mathcal{H}_{p,n}$ is a 
cuspidal eta quotient for all prime $p$ and all $n\in\N$.\qed

\section{
Irreducibility of $\mathcal{H}_{p,n}$}
\label{12:40pm}

Given two holomorphic eta quotients $h$ and $h'$ on $\Gm_0(p^n)$ such that 
\begin{equation}\cl{H}_{p, n}=h\times h',\label{6:28PM}\end{equation}
we show that except for the cases in \eqref{redH}, $h\in\{1,\cl{H}_{p,n}\}$. So, neither $h$ nor~$h'$ is a proper factor of $\cl{H}_{p,n}$.
The irreducibility of $\cl{H}_{p,n}$ now follows from
Theorem~\hyperlink{E}{E}.

\begin{lem}
Let $p$ be a prime, $n\in\N$ and let
$\cl{F}_{p^n}$ be as defined in \eqref{5:25AM}. Unless $p^n=2$, the eta quotient
${\mathcal{F}_{p^n}}/{\mathcal{H}_{p,n}}$ is cuspidal. 
\label{6:18PM}\end{lem}
\begin{proof}
From \eqref{5:25AM} and \eqref{equ1}, it follows that 
\begin{equation}
  \mathcal{F}_{p^n}=\begin{cases}\displaystyle{(\e\e_p)^{p-1}}&\text{ if $n=1$}\\ \\
  \displaystyle{\e_p^{p^2-1}}&\text{ if $n=2$}\\ 
  \displaystyle{(\e_p\e_{p^{n-1}})^{p^2-p}\prod_{j=2}^{n-2}\e_{p^j}^{(p-1)^2}}&\text{ if $n\geq3$.}
                    \end{cases}
\end{equation}
Hence, from \eqref{hpn} we obtain
\begin{equation}
  \frac{\mathcal{F}_{p^n}}{\mathcal{H}_{p,n}}=\begin{cases}\displaystyle{(\e\e_p)^{p-2}}&\text{ if $n=1$}\\ \\
    \displaystyle{(\eta\eta_{p^n})^{p-1}\eta_p^{p-2}\prod_{s=1}^{n/2-1}\frac{\eta_{p^{2s-1}}\eta_{p^{2s+1}}^{p-1}}{\eta_{p^{2s}}}}&\text{ if $n\geq2$ is even.}\\
 \displaystyle{(\e\e_{p^n})^{p-1}(\e_p\e_{p^{n-1}})^{p-2}\prod_{j=2}^{n-2}\e_{p^j}^{p-1}}&\text{ if $n\geq3$ is odd.}
                    \end{cases}
\label{12:26PM}\end{equation}
Since for all $s\in\N$, the eta quotient ${\eta_{p^{2s-1}}\eta_{p^{2s+1}}^{p-1}}/{\eta_{p^{2s}}}$ is a rescaling of ${\eta\eta_{p^2}^{p-1}}/{\eta_p}$, which is 
a holomorphic eta quotient\footnote{Since $A_{p^2}$ is diagonally dominant (see~\eqref{23July}) and since $p-1\geq1$, the orders of vanishing of ${\eta\eta_{p^2}^{p-1}}/{\eta_p}$
at both the cusps $1$ and $1/p^2$ of $\Gamma_0(p^2)$ (corresponding to the first and last rows of $A_p^2$) are nonnegative. Hence, we only require to 
check the nonnegativity of the order of vanishing of ${\eta\eta_{p^2}^{p-1}}/{\eta_p}$ at the cusp $1/p$, which is $(1-p+(p-1))/24=0$ (see~\eqref{27.04.2015.1}).}
of level $p^2$, and since the cuspidal eta quotient $\e$ is a factor of ${\mathcal{F}_{p^n}}/{\mathcal{H}_{p,n}}$ except for the case $p^n=2$, it follows that unless\footnote{${\mathcal{F}_{2}} = {\mathcal{H}_{2,1}}$.} $p^n=2$, the eta quotient ${\mathcal{F}_{p^n}}/{\mathcal{H}_{p,n}}$ is cuspidal. 
\end{proof}

\begin{lem}For a prime $p$ and $n\in\N$ with $p^n>2$,
let $X,X',X''\in\Z^{\bb{D}_{p^n}}$ be such that 
\begin{equation}\e^X=\mathcal{H}_{p,n},  \ \e^{X'}=h \text{ \ and \ } \e^{X''}=h',
\label{9:44AM}\end{equation}
where $h$ and $h'$ are as in \eqref{6:28PM}. Then 
$X,X',X''\in B_{p^n}\cdot[0,1)^{\bb{D}_{p^n}},$
where 
$B_{p^n}\in\Z^{\bb{D}_{p^n}\times\bb{D}_{p^n}}$ is as defined in \eqref{equ1}. 
\label{10:16PM}\end{lem}
\begin{proof} 
Let $,Y,Y',Y''\in\Q^{\bb{D}_{p^n}}$ be such that\footnote{The invertibility of $B_{p^n}$ follows from \eqref{28.08.2015A}} 
\begin{equation}
 Y=B^{-1}_{p^n}X, \ Y'=B^{-1}_{p^n}X' \text{ \ and  \ }Y''=B^{-1}_{p^n}X''.
\label{11:58AM}
\end{equation}
Since each row of $B^{-1}_{p^n}$ is a positive multiple of a row of $A_{p^n}$ (see~\eqref{10:39AM} and \eqref{mtn}) and since $\cl{H}_{p,n}=h\times h'$
where both $h$ and $h'$ are holomorphic, it follows from \eqref{9:44AM} and \eqref{11:58AM} that
\begin{equation}
Y=Y'+Y'',
\nonumber\end{equation}
where  all the entries of $Y,Y'$ and $Y''$ are nonnegative (see~\eqref{28.04}). Hence, it suffices to show that  $Y\in[0,1)^{\bb{D}_{p^n}}.$
Since for $p^n>2$, both $\mathcal{H}_{p,n}$ and ${\mathcal{F}_{p^n}}/{\mathcal{H}_{p,n}}$ are cuspidal,
it follows from Corollary~\ref{5:46AM} and \eqref{11:58AM} that $Y\in(0,1)^{\bb{D}_{p^n}}$.
\end{proof}

From Lemma~\ref{10:16PM} and Lemma~\ref{00:36AM} we obtain the following.

\begin{co}
Let $p$ be a prime and $n\in\N$ with $p^n>2$. Let $\cl{S}_{p^n}$ be a special set of cusp classes of $\Gamma_0(p^n)$. Then each factor of the eta quotient 
$\cl{H}_{p,n}$ is  uniquely identified by its orders of vanishing at the cusp classes in 
 $\cl{S}_{p^n}$.
\label{9:53PM}\end{co}

\begin{lem}Let 
$p$ be a prime, $n\in\N$ and 
let $f$ be an eta quotient on $\Gm_0(p^n)$. Let $a:=24\cdot\ord_{1}(f;\Gm_0(p^n))$ and $b:=24\cdot\ord_{1/p^n}(f;\Gm_0(p^n))$. Then 
$$(p^{2-r_n}-1)\hs|\hs(a-b)$$
where $r_n\in\{0,1\}$ is the residue of $n$ modulo 2.
\label{11:37PM}\end{lem}
\begin{proof}Let $X\in\Z^{\bb{D}_{p^n}}$ be such that $f=\e^X$.
It follows from \eqref{28.04} and \eqref{12:47PM} that 
$$a=\sum_{j=0}^np^{n-j}X_{p^j}\text{ \ and \ }b=\sum_{j=0}^np^{j}X_{p^j}.$$
The claim now follows from the fact that $j$ and $n-j$ have the same parity 
if and only if $n$ is even.
\end{proof}

\

{\flushleft\textit{Proof of the irreducibility of $\cl{H}_{p,n}$ except for the cases in \eqref{redH}.}} 
Theorem~\hyperlink{E}{E} implies that it suffices
to prove that $\cl{H}_{p,n}$ is not factorizable on~$\Gm_0(p^n)$.
\vspace{3pt}

{\flushleft\textbf{Case I.} ($n$ even)}
\vspace{2pt}

Let $h$ be a factor of $\cl{H}_{p,n}$ on $\Gm_0(p^n)$, let $a:=24\cdot\ord_{1}(h;\Gm_0(p^n))$ and $b:=24\cdot\ord_{1/p^n}(h;\Gm_0(p^n))$. 
Since $1/p^n$ and $\infty$
represent the same cusp of $\Gamma_0(p^n)$ and since $\cl{H}_{p,n}$ is a Fricke eigenform of level $p^n$, it follows from \eqref{ordinfty} and \eqref{11:03AM} that 
both the nonnegative integers $a$ and $b$ are bounded above by~$1$. So, from
Lemma~\ref{11:37PM}, we obtain $a=b$. Now, it follows from Corollary~\ref{9:53PM} that
\begin{equation}
h(z)=\begin{cases}1&\text{if $a=b=0$}\\
\cl{H}_{p,n}(z)&\text{otherwise.}
\end{cases}
\label{00:42AM}\end{equation}
So, $\cl{H}_{p,n}$ is irreducible for each prime $p$ and even $n$.
\vspace{3pt}

{\flushleft\textbf{Case II.} ($np$ odd,  $n>3$)}
\vspace{2pt}

Since $n>3$ is odd, it follows from Case I that $\cl{H}_{p,n-1}$ is irreducible.
Since $\cl{H}_{p,n-1}$ is a Fricke eigenform of level $p^{n-1}$, for all $M\in p^{n-1}\Z$, every Atkin-Lehner involution of $\cl{H}_{p,n-1}$ on $\Gm_0(M)$ is just a rescaling of $\cl{H}_{p,n-1}$. In particular, \eqref{hpn} and \eqref{11:59AM} imply that
\begin{equation}\mathcal{AL}_{p^n,p^n}(\cl{H}_{p,n-1}(z))=\cl{H}_{p,n-1}(pz).
\label{12:16PM}\end{equation}
Since $\cl{H}_{p,n-1}$ is irreducible, it follows from Lemma~\ref{7:03PM} and Theorem~\hyperlink{E}{E} that $\cl{H}_{p,n-1}(pz)$ is also irreducible\footnote{See Corollary~$\ref{cor2}$.}. 

Since $n$ is odd and since $1/p^n$ and $\infty$ represent the same cusp of $\Gm_0(p^n)$, it follows from 
\eqref{ordinfty} that 
\begin{equation}
\ord_{1/p^n}(\cl{H}_{p,n-1}(z);\Gm_0(p^n))=\frac1{24}, \ 
\ord_{1/p^n}(\cl{H}_{p,n-1}(pz);\Gm_0(p^n))=\frac{p}{24}. 
\label{2:36PM}
\end{equation}
Hence, from \eqref{12:16PM}, \eqref{11:03AM} and 
\eqref{2:36PM}, 
we obtain
 \begin{equation}
 \ord_{1}(\cl{H}_{p,n-1}(pz);\Gm_0(p^n))=\frac{1}{24}, \
 \ord_{1}(\cl{H}_{p,n-1}(z);\Gm_0(p^n))=\frac{p}{24}.
 \label{2:50PM}
\end{equation}

Let $h$ be a factor of $\cl{H}_{p,n}$ on $\Gm_0(p^n)$, let $a:=24\cdot\ord_{1}(h;\Gm_0(p^n))$ and $b:=24\cdot\ord_{1/p^n}(h;\Gm_0(p^n))$. 
Since $1/p^n$ and $\infty$
represent the same cusp of $\Gamma_0(p^n)$ and since $\cl{H}_{p,n}$ is a Fricke eigenform of level $p^n$, it follows from \eqref{ordinfty} and \eqref{11:03AM} that 
both the nonnegative integers $a$ and $b$ are bounded above by $p+1$. Since $p>2$,
Lemma~\ref{11:37PM} implies that unless $a=b\in\{0,p+1\}$ (which yields the same results as in \eqref{00:42AM}), there are only two possibilities, viz. $a=p,b=1$ or $a=1,b=p$. 
Since both $\cl{H}_{p,n-1}(z)$ and $\cl{H}_{p,n-1}(pz)$ are irreducible and 
since $\{1,1/p^n\}$ represents
a special set of cusp classes  
of $\Gamma_0(p^n)$, it follows from \eqref{2:36PM}, \eqref{2:50PM}, Corollary~\ref{9:58AM}, Lemma~\ref{10:16PM} and Lemma~\ref{00:36AM}
that
\begin{equation}
h(z)=\begin{cases}\cl{H}_{p,n-1}(z)&\text{if $a=p$ and $b=1$}\\
\cl{H}_{p,n-1}(pz)&\text{otherwise,}
\end{cases}
\label{00:58AM}\end{equation}
i.e., each proper factor of $\cl{H}_{p,n}(z)$ is either $\cl{H}_{p,n-1}(z)$ or $\cl{H}_{p,n-1}(pz)$.
Now, it follows from \eqref{00:58AM}, \eqref{2:36PM}, \eqref{2:50PM} and \eqref{ordinfty} that
$$\cl{H}_{p,n}(z)=\cl{H}_{p,n-1}(z)\hs\cl{H}_{p,n-1}(pz).$$
However, \eqref{hpn} implies that for each prime $p$ and for all odd $n>3$, we have
\begin{equation}
    \frac{\cl{H}_{p,n-1}(z)\hs\cl{H}_{p,n-1}(pz)}{\cl{H}_{p,n}(z)}=\frac{\big(\e(p^2z)\e(p^{n-2}z)\big)^{p^2-p+1}}{\big(\e(pz)\e(p^{n-1}z)\big)^p}.
\nonumber\end{equation}
Thus, we get a contradiction\footnote{See~\eqref{6:50PM}.}!
Hence, $\mathcal{H}_{p,n}$ is irreducible except for the cases in \eqref{redH}.
\qed

\section{The matrix identities}
\label{mi}

We provide unimodular matrices $U_{p^n}$ and $V_{p^n}\in\Z^{\bb{D}_{p^n}\times\bb{D}_{p^n}}$ along with their inverses $U'_{p^n}$ and $V'_{p^n}$ such that $D_{p^n}=U_{p^n}B_{p^n}V_{p^n},$
where $B_{p^n}$ is as in (\ref{equ1}) and $D_{p^n}$ is as in (\ref{6:48AM}).
\vspace*{5pt}
\newline For $n=1$, define 
\begin{alignat*}{3}
U_{p}&:=\left(\begin{array}{rr}
0 & -1 \\
1 & p
\end{array}\right),\text{  }
&V_{p}&:=\left(\begin{array}{rr}
1 & p \\
0 & 1
\end{array}\right),\\
U'_{p}&:=\left(\begin{array}{rr}
p & 1 \\
-1 & 0
\end{array}\right)\text{ \hspace{2pt}and\hspace{2pt} }
&V'_{p}&:=\left(\begin{array}{rr}
1 & -p \\
0 & 1
\end{array}\right).
\end{alignat*}
For $n=2$, define 
\begin{alignat*}{3}
U_{p^2}&:=\begin{pmatrix}
0& -1& -p\\
0& 0& -1\\
1& p& p^2
  \end{pmatrix},\text{  }
&V_{p^2}&:=\begin{pmatrix}
1& 1& p^2\\
0& 1& 1\\
0& 0& 1
  \end{pmatrix},\\
U'_{p^2}&:=\begin{pmatrix}
p& 0& 1\\
-1& p& 0\\
0& -1& 0
  \end{pmatrix}\text{ \hspace{2pt}and\hspace{2pt} }
  &V'_{p^2}&:= \begin{pmatrix}
1& -1& -(p^2-1)\\
0& 1& -1\\
0& 0& 1
  \end{pmatrix}.
\end{alignat*}
\vspace*{5pt}
For $n=3$, define
\begin{alignat*}{3}U_{p^3}&:=\begin{pmatrix}
0&-1&-p&-p^2\\
0&0&-1&-p\\
0&0&-p&-(p^2+1)\\
1 &p &p^2 &p^3
  \end{pmatrix},\text{  }
&V_{p^3}&:=\begin{pmatrix}
1&0&1&p^3\\
0&0&1&p\\
0&-1&p&1\\
0&0&0&1
  \end{pmatrix},\\
\\
U'_{p^3}&:=\begin{pmatrix}
p&0&0&1\\
-1&p&0&0\\
0&-(p^2+1)&p&0\\
0&p&-1&0
  \end{pmatrix}\text{ \hspace{2pt}and\hspace{2pt} }
&V'_{p^3}&:= \begin{pmatrix}
1 &-1 &0 &-p(p^2-1)\\
0 &p &-1 &-(p^2-1)\\
0&1&0&-p\\
0&0&0&1
\end{pmatrix}.
\end{alignat*}
  \vspace*{5pt}
\newline For $n>3$, define $U_{p^n}=(U_{p^n}(p^i,p^j))_{\hs0\leq i,j\leq n}$ by
  \vspace*{2pt}
\begin{equation}\begin{tabular}{ l|c|c| }
\multicolumn{1}{r}{}
 &  \multicolumn{1}{c}{$j=0$}
 & \multicolumn{1}{c}{$j>0$}\\
\cline{2-3}
\vspace*{-.45cm}
&&\\
$i<n-1$ & 0 & $\left\{\begin{array}{cl}
-p^{j-i-1}&\text{if  $j>i$}\\  
0&\text{otherwise}    
      \end{array}\right.$\\
\cline{2-3}
\vspace*{-.45cm}
&&\\
$i=n-1$ & 0 & $-p^{n-j}\cdot\frac{p^{2(j-1)}-1}{p^2-1}$ \\
\cline{2-3}
$i=n$ &1 & $p^j$\\
\cline{2-3}
\end{tabular}\hspace{4pt}, 
\label{u}\end{equation}

\vspace{.2cm}\quad define $V_{p^n}=(V_{p^n}(p^i,p^j))_{\hs0\leq i,j\leq n}$ by
  \vspace*{2pt}
\begin{equation}\begin{tabular}{ l|c|c|c|c| }
\multicolumn{1}{r}{}
 & \multicolumn{1}{c}{$j=0$}
 &  \multicolumn{1}{c}{$0<j<n-1$}
 & \multicolumn{1}{c}{$j=n-1$} 
 & \multicolumn{1}{c}{$j=n$}\\
\cline{2-5}
$i=0$&1 &0&1&$p^n$\\
\cline{2-5}
\vspace*{-.45cm}
&&&&\\
$0<i<n$ & 0& $\left\{\begin{array}{cl}
-p^{i-j-1}&\text{if  $i>j$}\\  
0&\text{otherwise}    
      \end{array}\right.$&$p^{i-1}$&$p^{n-i-1}$\\
\cline{2-5}
$i=n$ &0 & $0$&0&1\\
\cline{2-5}
\end{tabular}\hspace{4pt},
\label{v}\end{equation}

\quad define $U'_{p^n}=(U'_{p^n}(p^i,p^j))_{\hs0\leq i,j\leq n}$ by
\begin{equation*}\begin{tabular}{ l|c|c|c|c|c| }
\multicolumn{1}{r}{}
 & \multicolumn{1}{c}{$j=0$}
 &  \multicolumn{1}{c}{$0<j<n-2$}
 & \multicolumn{1}{c}{$j=n-2$} 
 & \multicolumn{1}{c}{$j=n-1$}
 & \multicolumn{1}{c}{$j=n$}
\\
\cline{2-6}
$i=0$&$p$ &0&0&0&1\\
\cline{2-6}
$0<i<n-1$ & $\begin{matrix}-1\hspace{8pt}\\0\\ \vdots\\0\end{matrix}$& $\left\{\begin{array}{cl}
p&\text{if  $i=j$}\\
-1\hspace{8pt}&\text{if  $i=j+1$}\\
0&\text{otherwise}    
      \end{array}\right.$&$\begin{matrix}0\\ \vdots\\0\\p\end{matrix}$&0&0\\
\cline{2-6}
$i=n-1$ & 0 & $-p^{n-j}$&$-(p^2+1)$&$p$&0\\
\cline{2-6}
$i=n$ &0 &$p^{n-j-1}$&$p$&$-1\hspace{8pt}$&0\\
\cline{2-6}
\end{tabular}\hspace{2.5pt},
\end{equation*}
\quad and define $V'_{p^n}=(V'_{p^n}(p^i,p^j))_{\hs0\leq i,j\leq n}$ by
\begin{equation}\begin{tabular}{ l|c|c|c| }
\multicolumn{1}{r}{}
 & \multicolumn{1}{c}{$j=0$}
 &  \multicolumn{1}{c}{$0<j<n$}
 & \multicolumn{1}{c}{$j=n$}\\
\cline{2-4}
$i=0$&1 &$\begin{matrix}-1\quad&0\quad&\cdots\quad&0\quad\end{matrix}$&$-p^{n-2}(p^2-1)$\\
\cline{2-4}
$0<i<n-1$ & 0& $\left\{\begin{array}{cl}
p&\text{if  $i=j$}\\
-1\hspace{8pt}&\text{if  $i=j-1$}\\
0&\text{otherwise}    
      \end{array}\right.$&$-p^{n-i-2}(p^2-1)$\\
\cline{2-4}
$i=n-1$ & 0 & $\begin{matrix}1\quad&0\quad&\cdots\quad&0\quad\end{matrix}$&$-p^{n-2}$\\
\cline{2-4}
$i=n$ &0 & $0$&1\\
\cline{2-4}
\end{tabular}\hspace{4pt}. 
\label{v'}\end{equation}

\begin{lem} For a prime $p$ and $n\in\bb{N}$, let $U_{p^n},V_{p^n}, U'_{p^n}$ and $V'_{p^n}\in\Z^{\bb{D}_{p^n}\times\bb{D}_{p^n}}$ be the matrices as defined above. Let $B_{p^n}$ be as in $(\ref{equ1})$, $D_{p^n}$ be as in $(\ref{6:48AM})$
and let $I_{p^n}$ be the identity matrix in $\Z^{\bb{D}_{p^n}\times\bb{D}_{p^n}}$. Then 
$$U_{p^n}U'_{p^n}=I_{p^n},\hspace{3pt} V_{p^n}V'_{p^n}=I_{p^n}\hspace{3pt}\text{ and } \hspace{3pt}D_{p^n}=U_{p^n}B_{p^n}V_{p^n}.$$
\end{lem}

\begin{proof}
For $n\leq 3$, the claim holds trivially. For $n>3$, the proofs of the 
equalities of the corresponding matrix entries 
in each of these matrix relations
involve at most summation of some geometric series. For example,
consider the identity $D_{p^n}=U_{p^n}B_{p^n}V_{p^n}$.
 It is equivalent to $D_{p^n}V'_{p^n}=U_{p^n}B_{p^n}$, assuming  $V_{p^n}V'_{p^n}=I_{p^n}$. Now  from \eqref{6:48AM} and \eqref{v'},  we see that for $i,j\in\{0,\hdots,n\}$, the $(i,j)$-th entry of $D_{p^n}V'_{p^n}$ is given by
\begin{equation*}\begin{tabular}{ l|c|c|c| }
\multicolumn{1}{r}{}
 & \multicolumn{1}{c}{$j=0$}
 &  \multicolumn{1}{c}{$0<j<n$}
 & \multicolumn{1}{c}{$j=n$}\\
\cline{2-4}
$i=0$&1 &$\begin{matrix}-1\quad&0\quad&\hdots\quad&0\quad\end{matrix}$&$-p^{n-2}(p^2-1)$\\
\cline{2-4}
$0<i<n-1$ & 0& $\left\{\begin{array}{cl}
p&\text{if  $i=j$,}\\
-1\hspace{8pt}&\text{if  $i=j-1$,}\\
0&\text{otherwise.}    
      \end{array}\right.$&$-p^{n-i-2}(p^2-1)$\\
\cline{2-4}
$i=n-1$ & 0 & $\begin{matrix}p^{n-1}\quad&0\quad&\hdots\quad&0\quad\end{matrix}$&$-p^{2n-3}$\\
\cline{2-4}
$i=n$ &0 & $0$&$p^{n-1}(p^2-1$)\\
\cline{2-4}
\end{tabular}\hspace{8pt}.\end{equation*}
If we consider the case $0<i<n-1$ and $0<j<n$, then from
$(\ref{u})$ and $(\ref{equ1})$ it follows that the product of $U_{p^n}(p^i,\hs\un)$, i.e. the $i$-th row of $U_{p^n}$
and $B_{p^n}(\un\hs,p^j)$, i.e. the $j$-th column of $B_{p^n}$ is \begin{align*}-\sum_{k=i+1}^np^{k-i-1}\hspace{2pt}B_{p^n}(p^k,p^j)&=\sum_{k=i+1}^np^{k-i-1}(p\hspace{1pt}\delta_{|k-j|\hspace{1pt},\hspace{1pt}1}-(p^2+1)\hspace{1pt}\delta_{k,\hspace{1pt}j})\\
&=\sum_{k=\max\{i+1,\hspace{1pt}j-1\}}^{j+1}p^{k-i-1}(p\hspace{1pt}\delta_{|k-j|\hspace{1pt},\hspace{1pt}1}-(p^2+1)\hspace{1pt}\delta_{k,\hspace{1pt}j})\\
&=\begin{cases}
\hspace{8pt}0&\text{ if $j<i$,}\\
\hspace{8pt}p&\text{ if $j=i$,}\\
-1&\text{ if $j=i+1$,}\\
\hspace{8pt}0&\text{ if $j>i+1$,}
\end{cases}
\end{align*}
where $\delta$ is the Kronecker delta function. So, the claim holds in this case.

Again, for $i=n-1$ and $0<j<n$, we have
\begin{align*}
U_{p^n}(p^{n-1},\hs\un)\cdot &B_{p^n}(\un\hs,p^j)\\
&=
-\sum_{k=1}^np^{n-k}\cdot\frac{p^{2(k-1)}-1}{p^2-1}\cdot\hspace{2pt}B_{p^n}(p^k,p^j)\\
&=\sum_{k=\max\{2,\hspace{1pt}j-1\}}^{j+1}p^{n-k}\cdot\frac{p^{2(k-1)}-1}{p^2-1}\cdot\left(p\hspace{1pt}\delta_{|k-j|\hspace{1pt},\hspace{1pt}1}-(p^2+1)\hspace{1pt}\delta_{k,\hspace{1pt}j}\right)\\
&=\begin{cases}
p^{n-1}&\text{ if $j=1$,}\\
0&\text{ otherwise,}
\end{cases}
\end{align*}
Thus, the claim also holds in this case.
The rest of the proof is analogous and requires only routine verifications as above.
\end{proof}

\section*{Conjectures and open problems}

Theorem~\ref{prp1} provides us with an example of a primitive and irreducible holomorphic eta quotient which is a Fricke eigenform of level $p^n$, except for the levels
appearing in \eqref{redH}, viz. primes, primes cubed and the odd powers of $2$. 
The nonexistence of such an eta quotient of a prime level is trivial. We conjecture that the same assertion also holds in the other two cases:

\begin{conj}
For $n$ odd, there does not exist any  primitive and irreducible holomorphic eta quotient which is a Fricke eigenform of level $2^n$.
\label{00:00AM}\end{conj}

Further, for cube of primes, even the condition of invariance under Fricke involution may be removed:

\begin{conj}For each prime $p$, 
there does not exist any primitive and irreducible holomorphic eta quotient of level $p^3$.
\label{00:13AM}\end{conj}

The above conjecture, which asserts the nonexistence of primitive and irreducible holomorphic eta quotients of certain levels,
contrasts well with Corollary~\ref{snld}, which 
states that for each $k\in\N$, there exists such an eta quotient of weight $k/2$.
In particular, Theorem~6.4 in \cite{B0} and the above conjecture together imply Conjecture~2.(a) of \cite{B3}.

\begin{conj}
Given two irreducible holomorphic eta quotients $f=\prod_{d|M}\e_d^{X_d}$ and $g=\prod_{\delta|N}\e_{\delta}^{Y_{\delta}}$ respectively on $\Gm_0(M)$ and $\Gm_0(N)$ with mutually coprime $M$ and~$N$, the holomorphic eta quotient $f\circledast g$,
defined by
\begin{equation}
    f\circledast g:=\prod_{d|M}\prod_{\delta|N}\e_{d\delta}^{X_dY_{\delta}}
\nonumber\end{equation}
is irreducible. 
\label{10:56PM}\end{conj}

The truth of the above conjecture would imply the irreducibility\footnote{See Theorem~\ref{prp1}.} of 
\begin{equation}
    \underset{\substack{p^n\| N\\ p\text{ prime}}}{\circledast}\cl{H}_{p,n},
\end{equation}
if no unitary divisor of $N$ is prime or cube of a prime or an odd power of 2.

 Since there are only $\mathrm{d}(N)$ cusp classes of $\GN$, 
 there are at most $\mathrm{d}(N)$ axial irreducible holomorphic eta quotients on $\GN$. 
Lemma~6 in \cite{B3} lists $\mathrm{d}(N)$ axial holomorphic eta quotients on $\GN$, none of which is factorizable on $\GN$ and Lemma~7 in \cite{B3} proves the irreducibility of some of those. We conjecture that all the axial holomorphic eta quotients given in Lemma~6 in~\cite{B3} are irreducible:
\begin{conj}
 There are exactly $\mathrm{d}(N)$ 
 axial irreducible holomorphic eta quotients on $\GN$.
\label{11:20PM}\end{conj}
It follows from Theorem~\hyperlink{E}{E} and Lemma~6 in \cite{B3} that the above conjecture holds in the case where $N$ is a prime power. The general case of the above conjecture then follows from Conjecture~\ref{10:56PM}. We note that Conjecture~\ref{11:20PM} also follows from
Lemma~6 in \cite{B3} and Conjecture~1 in \cite{B5}.

 \section*{Acknowledgments}
 The author is thankful to Don Zagier and Tomoyoshi Ibukiyama 
 for their comments. The numerical experiments for this paper were conducted with
 the open-source computer algebra system \textsf{PARI/GP} \cite{PARI2}.

\bibliography{lpe}
\bibliographystyle{IEEEtranS}
\nocite{*}
\end{document}